\documentclass[12pt]{amsart}
\usepackage{color}
\usepackage{amssymb}
\usepackage{comment}
\usepackage{pdfpages}
\usepackage{graphicx}% Include figure files
\usepackage{dcolumn}% Align table columns on decimal point
\RequirePackage[numbers]{natbib}
\RequirePackage[colorlinks=true, pdfstartview=FitV, linkcolor=blue,
  citecolor=blue, urlcolor=blue]{hyperref}
\RequirePackage{hypernat}
\usepackage{paralist}

\usepackage{graphicx}
\graphicspath{{./figures/}}
\usepackage{amsfonts}
\usepackage{amsmath}
\usepackage{amsthm}
\usepackage{amssymb}
\usepackage{amsbsy}
\usepackage{epsfig}
\usepackage{fullpage}
\usepackage{natbib, mathrsfs} 
\usepackage{verbatim}
\usepackage[latin1]{inputenc}
\usepackage{mhequ}
\usepackage{algorithm}
\usepackage{algorithmic}

\numberwithin{equation}{section}

\usepackage{tikz,calc} %for pictures
\usepackage{dcolumn}
\usetikzlibrary{arrows,positioning}
\tikzset{
    >=stealth',
    pil/.style={
           ->,
           thick,
           shorten <=2pt,
           shorten >=2pt,}
}

\def \be{\begin{equs}}
\def \ee{\end{equs}}
\def \P{\mathbb{P}}
\def \E{\mathbb{E}}

\newcommand \TV{\mathrm{TV}}

\def \tmix{\tau_{\mathrm{mix}}}

\def \TV{\mathrm{TV}}

\def \L{\mathcal{L}}
\newcommand{\tRe}{\tau_{\mathrm{ret}}}
\newcommand{\vmax}{V_{\mathrm{max}}}
\newcommand{\tEsc[1]}{\tau_{#1,\mathrm{esc}}}

\newtheorem{theorem}{Theorem}[section]
\newtheorem{lemma}[theorem]{Lemma}
\newtheorem{remarks}[theorem]{Remarks}

\newtheorem{defn}[theorem]{Definition}

\theoremstyle{plain}
\newtheorem{thm}{Theorem}
\newtheorem*{thm-non}{Theorem}
\newtheorem{cor}[thm]{Corollary}

\theoremstyle{definition}
\newtheorem{example}[theorem]{Example}
\newtheorem{remark}[theorem]{Remark}
\begin{document}

%\begin{frontmatter}
% "Title of the paper"
\title[Elementary Bounds on Mixing Times for Decomposable Reversible Markov Chains]
{Elementary Bounds on Mixing Times for Decomposable Markov Chains}

% indicate corresponding author with \corref{}
% \author{\fnms{John} \snm{Smith}\corref{}\ead[label=e1]{smith@foo.com}\thanksref{t1}}
 %\thankstext{t1}{Thanks to somebody} 
 %\address{line 1\\ line 2\\ printead{e1}}
% \affiliation{Some University}

\author{Natesh S. Pillai$^{\ddag}$}
\thanks{$^{\ddag}$pillai@fas.harvard.edu, 
   Department of Statistics
    Harvard University, 1 Oxford Street, Cambridge
    MA 02138, USA}

\author{Aaron Smith$^{\sharp}$}
\thanks{$^{\sharp}$smith.aaron.matthew@gmail.com, 
   Department of Mathematics and Statistics
University of Ottawa, 585 King Edward Drive, Ottawa
ON K1N 7N5, Canada}
 % \thanks{NSP is partially supported by NSF-DMS 1107070}

\maketitle

% AOS,AOAS: If there are supplements please fill:
%\begin{supplement}[id=suppA]
%  \sname{Supplement A}
%  \stitle{Title}
%  \slink[url]{http://lib.stat.cmu.edu/aoas/???/???}
%  \sdescription{Some text}
%\end{supplement}

\begin{abstract}
Many finite-state reversible Markov chains can be naturally decomposed into ``projection" and ``restriction" chains. In this paper we provide bounds on the total variation mixing times of the original chain in terms of the mixing properties of these related chains. This paper is in the tradition of existing bounds on Poincar\'{e} and log-Sobolev constants of Markov chains in terms of similar decompositions \cite{JSTV04,MaRa02,MaRa06,MaYu09}. Our proofs are simple, relying largely on recent results relating hitting and mixing times of reversible Markov chains \cite{PeSo13,Oliv12b}. We describe situations in which our results give substantially better bounds than those obtained by applying existing decomposition results and provide examples for illustration.
\end{abstract}

\section{Introduction}

In this paper, we study the rate of convergence to stationarity of an irreducible, aperiodic and reversible Markov chain with kernel $K$ on a finite state space $\Omega$ by decomposing $\Omega$ into subsets $\{  \Omega_{i} \}_{i=1}^{n}$. Our main results bound the total variation mixing time of $K$ in terms of the mixing times of the \textit{traces} (or \textit{restrictions})  $K_{i}$ of $K$ on each subset $\Omega_{i} \subset \Omega$, combined with some information on the mixing of $K$ between the subsets. This latter mixing time is studied through the construction of a \textit{projected} kernel $\overline{K}$ on $\{ 1,2,\ldots, n\}$. Although the details of our constructions differ, this general approach is not new: it has been the subject of a number of papers \cite{MaRa00,JSTV04,MaRa02,MaRa06, MaYu09} and has been successfully applied to many problems (see, \textit{e.g.}, 
 \cite{DLP10,kovchegov2015path}).

Our results, like those in \cite{MaRa00,JSTV04,MaRa02,MaRa06, MaYu09}, are useful in the common situation that a complicated Markov chain is hard to study directly, but is composed of smaller pieces that are easier to study in isolation or have already been studied. Our main goal is to provide bounds for the mixing time that are easy to apply in a wide range of applications. Our main results are based on the remarkable results of \cite{PeSo13,Oliv12b}, where the authors derived an upper bound on the mixing times of reversible Markov chains in terms of their hitting times. 
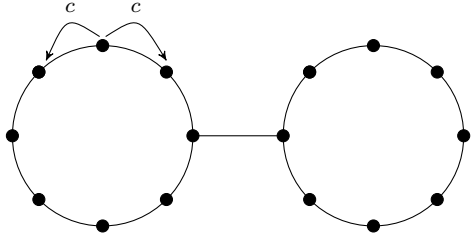
\begin{figure}
\begin{tikzpicture}[scale=0.4]
\draw (2 , 2) circle [radius = 3cm];
\draw (11 , 2) circle [radius = 3cm];
\draw[fill] (5,2) circle [radius=0.20];
\draw[fill] (2,5) circle [radius=0.20];
\draw[fill] (4.12, 4.12)  circle [radius=0.20];
\draw[fill] (-0.12, 4.12)  circle [radius=0.20];
\draw[fill] (4.12, -0.12)  circle [radius=0.20];
\draw[fill] (-0.12, 4.12)  circle [radius=0.20];
\draw[fill] (-0.12, -0.12)  circle [radius=0.20];
\draw[fill] (-1,2) circle [radius=0.20];
\draw[fill] (2,-1) circle [radius=0.20];
\draw[fill] (8,2) circle [radius = 0.20];
\draw[fill] (11,5) circle [radius=0.20];
\draw[fill] (14,2) circle [radius=0.20];
\draw[fill] (11,-1) circle [radius=0.20];
\draw[fill] (	13.12, 4.12)  circle [radius=0.20];
\draw[fill] (9-0.12, 4.12)  circle [radius=0.20];
\draw[fill] (13.12, -0.12)  circle [radius=0.20];
\draw[fill] (9-0.12, 4.12)  circle [radius=0.20];
\draw[fill] (9-0.12, -0.12)  circle [radius=0.20];
\draw (5,2) -- (8,2);
\draw [->] (2.1,5.3) .. controls (3.2, 6) .. (4.12, 4.5)
        node[pos = 0.45, above]{\tiny{$c$}};
\draw [->] (1.9,5.3) .. controls (0.8, 6) .. (0.12, 4.5)
        node[pos = 0.45, above]{\tiny{$c$}};
 \end{tikzpicture}
\caption{PinceNez Graph with $m =8$}
\label{FigPinceNez} 
\end{figure} 

Our bounds are generally not comparable to earlier decomposition bounds, so we give a high level review of those results and explain some ways in which ours can be much better. Let $\varphi_{i}$, $\varphi_{i}^{\mathrm{rel}}$ be the mixing and relaxation times of $K_{i}$, let $\overline{\varphi}$, $\overline{\varphi}^{\mathrm{rel}}$ be the mixing and relaxation times of $\overline{K}$, and let $\varphi$, $\varphi^{\mathrm{rel}}$ be the mixing and relaxation times of $K$. The main results of earlier decomposition bounds all state, roughly, that
\be \label{IneqBasicDecomp}
\varphi^{\mathrm{rel}} = O \big( \overline{\varphi}^{\mathrm{rel}} \, \max_{1 \leq i \leq n} \varphi_{i}^{\mathrm{rel}} \big).
\ee  
The main innovation of \cite{JSTV04} allows this bound to be improved if $K$ satisfies certain regularity conditions. One of the consequences of Theorem 1 of \cite{JSTV04} is that, under certain regularity conditions, the upper bound \eqref{IneqBasicDecomp} can sometimes be replaced with the much smaller bound 
\be \label{IneqBasicDecompRegular}
\varphi^{\mathrm{rel}} = O \big( \max \big( \overline{\varphi}^{\mathrm{rel}}, \, \max_{1 \leq i \leq n} \varphi_{i}^{\mathrm{rel}} \big) \big).
\ee  
Unfortunately, many Markov chains of interest  that in fact satisfy \eqref{IneqBasicDecompRegular} for a natural decomposition do not satisfy the regularity conditions given as sufficient conditions in \cite{JSTV04} (see, \textit{e.g.,} the interacting particle systems we study in \cite{PiSm15}). Some of the main consequences of the results in this paper are new sufficient conditions under which a bound similar to \eqref{IneqBasicDecomp} can be replaced by a stronger bound similar to \eqref{IneqBasicDecompRegular}, and which can be used to obtain stronger bounds on the mixing times of interacting particle systems and other Markov chains. Although we give this overview in terms of relaxation times, all of our main result bounds the mixing time of $K$ in terms of the mixing times of $K_i$ and the occupation times of the original Markov chain on the sets $\Omega_i$, rather than the associated relaxation times.

Our new sufficient conditions can hold when those in \cite{JSTV04} do not, and the new cases that we cover include some important examples. A simple illustrative example is the symmetric random walk on $2m$-vertex ``Pince-nez" graph studied in Section 4.1 of \cite{JSTV04}. The $2m$-vertex Pince-nez graph consists of two copies of the $\mathbb{Z}_m$ connected by a single edge. The graph is pictured in Figure \ref{FigPinceNez} with $m = 8$. Consider a random walk on the $2m$-vertex Pince-nez graph which moves to a neighboring vertex with probability $c= O(1)$. It is natural to partition this graph into two copies of $\mathbb{Z}_m$. Using this partition, the authors in \cite{JSTV04} showed that the relaxation time of the random walk is $O(m^{3})$, which implies a bound of $O(m^{3} \log(m))$ on its mixing time. To our knowledge, no earlier works on decomposition bounds will give a bound better than $O(m^3)$ on the mixing time for this example. Our bounds can be used to show (see Section \ref{SubsecPinceNez}) that its mixing time (and thus its relaxation time) is $O(m^{2})$, which is indeed the correct order.

Of course, this ``Pince-nez" example is simple, and its mixing time can be derived using direct arguments such as coupling. We emphasize this example because it has two traits that are typical of chains and partitions for which our approach can improve on previous results: there are only a small number of `important' parts of the partition (\textit{i.e.,} $n$ is small compared to $\varphi_{i}^{\mathrm{mix}}$ and $\overline{\varphi}^{\mathrm{mix}}$), and the exit probabilities $K(x,\Omega_{2})$ from one part of the partition to another are very far from uniform in the initial point $x \in \Omega_{1}$. There are many interesting examples of Markov chains having these traits. \par

The motivation behind this paper was the study of interacting particle systems. Using the main results of this paper, we were able to resolve a conjecture of David Aldous on the mixing time of a `constrained'  Ising process on the lattice \cite{PiSm15}, up to a logarithmic factor. The kinetically constrained Ising process (KCIP) is described carefully in Example \ref{ExKcipCarefulStatement}, and Example \ref{ExKcipToy} describes a toy version of the KCIP that illustrates the key difference between our bound and that in \cite{JSTV04}.  We expect our bounds to be helpful in the study of other interacting particle systems with a varying number of particles. To explore this, in Section \ref{ExGenInt} we construct a large class of interacting particle systems, show that the bounds in \cite{JSTV04} cannot generally be tight, and explain why our bounds can be.

Our main application in Section \ref{SecRegAssump} gives another situation in which our bound is roughly of the form \eqref{IneqBasicDecompRegular} while the bounds in \cite{JSTV04} are roughly of the form \eqref{IneqBasicDecomp}. For both the KCIP and the family of examples in Section \ref{SecRegAssump}, our bounds allow us to recover the correct mixing time up to a factor that is logarithmic in the problem size, while the bounds from \cite{JSTV04} are off by a factor that is polynomial in the problem size and can be close to the square of the correct mixing time.

Our results can improve upon earlier bounds in other interesting situations, and can be worse than earlier bounds in others. See Section \ref{SubsecUser} for a brief overview, and Inequality \eqref{IneqVisSimilar} for a bound that is most visually similar to earlier decomposition bounds.

\subsection{Paper Overview}

After giving initial notation in Sections \ref{SubsecNotation}, \ref{SubsecProjRest} below, we give a `user's guide' to the paper in Section \ref{SubsecUser}. One of the main attractions of decomposition bounds such as those in \cite{JSTV04} is their ease of use: you can `plug in' estimates of certain familiar quantities related to the kernels $K_{i}$ and $\overline{K}$, such as their relaxation times, to obtain estimates for the kernel $K$. Our bounds are often similarly easy to use. However, the bounds in the different sections of our paper are written in terms of different quantities, and several of our intermediate bounds are written in terms of quantities that may not be familiar to the reader. Thus, it may not be immediately obvious which bounds are most relevant for a given problem. Our `user's guide' is designed to resolve this difficulty, allowing the reader to quickly obtain bounds that are written entirely in terms of familiar quantities, such as a mixing time or Lyapunov function. For several situations, the user's guide describes why our bounds may improve upon earlier bounds, refers to the most relevant and simplest bounds in our paper, and also refers to a prototypical worked example.

In Section \ref{SecGenMix}, we give our main lemmas and apply them to obtain an initial result that is visually similar to earlier decomposition bounds. In Section \ref{SecNaiveBoot} we introduce the notion of the well-covering time as well as a comparison theory for well-covering times; these are used to obtain decomposition bounds that are easier to compute than those in Section \ref{SecGenMix}. In Section \ref{SecSpecCond} we give much stronger decomposition bounds under certain regularity conditions and apply them. Finally, Section \ref{SecAppl} contains two applications. Auxiliary results and derivations are deferred to an Appendix.

\subsection{Basic Notation} \label{SubsecNotation}

We write $\mathbb{N} = \{0,1,2,\ldots\}$. For any $m \geq 1$, define $[m] = \{1,2,\cdots, m\}$.
For two positive functions $f,g$ on $\mathbb{N}$, we write $f = O(g)$ for $\sup_{m} \frac{f(m)}{g(m)} < \infty$, we write $f(m) = \Theta(g)$ if both $f = O(g)$ and $g = O(f)$, and we write $f=o(g)$ for $\limsup_{m \rightarrow \infty} \frac{f(m)}{g(m)} = 0$. For a sequence of positive random variables $\{ X_{m} \}_{m \geq 0}$ and a sequence of positive integers or random variables $\{Y_{m} \}_{m \geq 0}$, we say that $X_{m} = O(Y_{m})$ (or $X_{m} = O(Y_{m})$ with high probability) if for all $\epsilon > 0$ there exists a constant $C = C(\epsilon) < \infty$ so that $\limsup_{m \rightarrow \infty} \P[X_{m} > C Y_{m}] \leq \epsilon$. In both cases, we sometimes write ``$f$ is at least on the order of $g$'' to meant $g = O(f)$. We say that a random variable $X$ stochastically dominates a random variable $Y$ if $\P[X > s] \geq \P[Y > s]$ for all $s \in \mathbb{R}$. The letters $C,C_1$ \textit{etc.} will denote generic constants whose value may change from one occurence to the next but are independent of the problem size or $n$, the number of partitions.

For a monotonely increasing (but not necessarily injective) function $f: \mathbb{N} \mapsto \mathbb{N}$, 
\be
f^{-1}(x) \equiv \min \{ y \, : \, f(y) \geq x \}.
\ee
 For two distributions $\mu,\nu$ on a finite state space $\Omega$, the $L^{1}$, or total variation, distance between $\mu$ and $\nu$ is given by
\be 
\| \mu - \nu \|_{\TV} = \max_{A \subset \Omega} (\mu(A) - \nu(A)).
\ee 
The \emph{mixing profile} of a Markov chain $\{ X_t \}_{t \in \mathbb{N}}$ on $\Omega$ with stationary measure $\pi$ is defined as  
\be 
\tau(\epsilon) = \min \{ t > 0 \, : \, \max_{X_{0} = x \in \Omega} \| \mathcal{L}(X_{t}) - \pi \|_{\TV} < \epsilon \}
\ee
for all $0 < \epsilon < 1$. As usual, the \textit{mixing time} is defined by $\tmix = \tau(0.25)$. The dependence of a Markov chain $\{X_{t}\}_{t \in \mathbb{N}}$ on the initial conditions $X_0 =x$ is denoted by a subscript, \textit{e.g.}, $\E_x[\cdot] = \E[\cdot| X_0 = x]$.

\subsection{Projections and Restrictions of Markov  Kernels} \label{SubsecProjRest}

Let $\{ X_{t} \}_{t \in \mathbb{N}}$ be an irreducible, reversible and $\frac{1}{2}$-lazy Markov chain with kernel $K$ and stationary distribution $\pi$ on $\Omega$.
Our goal is to bound the mixing time of $\{ X_{t} \}_{t \in \mathbb{N}}$ in terms of the mixing times of various \textit{restricted} and \textit{projected} chains. Fix $n \in \mathbb{N}$ and let $\Omega = \sqcup_{i=1}^{n} \Omega_{i}$ be a partition of $\Omega$ into $n$ disjoint parts. Define the projection function $\mathcal{P}$ on $\Omega$ by
\be \label{IneqProfFunct}
\mathcal{P}(x) = \{ 1 \leq i \leq n \, : \, x \in \Omega_{i} \}.
\ee 
For $1 \leq i \leq n$, set $\eta_{i}(0) = -1$. Then, for $s \in \mathbb{N}$, recursively define the sequence of \textit{hitting} and \textit{occupation} times 
\be [EqDefOccupationMeasureCounters]
\eta_{i}(s) &= \min \{ t > \eta_{i}(s-1) \, : \, X_{t} \in \Omega_{i} \}, \\
\kappa_{i}(s) &= \max \{u \, : \, \eta_{i}(u) \leq s \} = \sum_{u=1}^s 1_{X_u \in \Omega_i}.
\ee 
Both $\eta, \kappa$ depend on the initial condition $X_0$.
We also define the associated \textit{restricted} processes $\{ X_{t}^{(i)} \}_{t \in \mathbb{N}}$ by
\be \label{EqDefRestChain}
X_{t}^{(i)} = X_{\eta_{i}(t)}. 
\ee 
This is also called the \textit{trace} of $\{ X_{t} \}_{t \in \mathbb{N}}$ on $\Omega_{i}$.
Since $\{ X_{t} \}_{t \in \mathbb{N}}$ is recurrent, we have for all $t \in \mathbb{N}$ that $\eta_{i}(t) < \infty$ almost surely, and so $X_{t}^{(i)}$ is almost surely well-defined for all $t \in \mathbb{N}$. The trace $\{ X_{t}^{(i)} \}_{t \in \mathbb{N}}$ is a Markov chain on $\Omega_{i}$, and we denote by $K_{i}$ the associated transition kernel on $\Omega_{i}$. The kernel $K_{i}$   inherits irreducibility, reversibility and $\frac{1}{2}$-laziness from $K$\footnote{ The transition kernel $K_{i}'$ corresponding to the notion of a restricted chain defined in \cite[Section 2]{JSTV04} satisfies $K_{i}'(x,y) \leq K_{i}(x,y)$ for all $x \neq y$. This means that our restricted chain is generally more rapidly mixing (\textit{e.g.}, $K_{i}$ dominates $K_{i}'$ in the sense of \cite{tierney1998note}), and it can be much more rapidly mixing (\textit{e.g.}, when $K$ is ergodic, $K_{i}$ is always ergodic and has a smaller mixing time than $K$, while $K_{i}'$ may not even be ergodic).} and its stationary distribution is given by $\pi_{i}(A) = \frac{ \pi(A)}{\pi(\Omega_{i})}$ for all $A \subset \Omega_{i}$. Let $\varphi_{i}$ be the mixing time of $K_{i}$ and set 
 \be \label{eqn:phimax}
 \varphi_{\max} = \max_{1 \leq i \leq n} \varphi_{i}.
 \ee 

We next define the \textit{projected} kernel $\overline{K}$. The state space of the projected chain is the set $[n]$ and its transition kernel is defined by  
\be \label{EqDefProjChain}
\overline{K}(i,j) = \frac{1}{\pi(\Omega_{i})}\sum_{x \in \Omega_{i}, y \in \Omega_{j}} \pi(x) K(x,y).
\ee

Throughout the paper, we are often interested in hitting times of various sets. We recall that, if $\{ X_{t} \}_{t \geq 0}$ is a Markov chain with state space $\Omega $ and $A \subset \Omega$, the hitting time $\tau_{A} \equiv \min \{ t \geq 0 \, : \, X_{t} \in A \}$ satisfies
\be \label{eqn:subexp}
\max_{x \in \Omega}  \P_x[\tau_{A} > k t] \leq \big( \max_{x \in \Omega}  \P_x[\tau_{A} >  t] \big)^{k}
\ee 
for all $k, t \in \mathbb{N}$. In particular, this implies
\be \label{eqn:subexp2}
\max_{x \in \Omega}  \P_x[ e^{-1} \frac{\tau_{A}}{\max_{y \in \Omega}\E_{y}[\tau_{A}]} > t] \leq e^{-t}.
\ee 
See inequality \eqref{IneqEndOfShortProof} in Appendix A for a short proof of \eqref{eqn:subexp2}. We use this fact throughout the paper, sometimes referring to the `subgeometric tails' of the hitting time distribution.

\subsection{User's Guide} \label{SubsecUser}

We give an overview of our results, with the aim of directing a reader to the most relevant bounds. Briefly, Lemma \ref{LemmaBasicMixing2} gives the strongest bounds in this paper, while Inequality \eqref{EqConclusionWellCoveringSimpleExample} gives mixing time bounds that are easiest to compute in terms of standard estimates.

We now describe several common situations in which our bounds can improve on those in the literature: 

\begin{itemize}
\item \textbf{Situation:} $n$ is small while $\max_{1 \leq i \leq n} \frac{\max_{x \in \Omega_{i}} \sum_{y \notin \Omega_{i}} K(x,y)}{\min_{x \in \Omega_{i}} \sum_{y \notin \Omega_{i}} K(x,y)}$ is large.
\begin{itemize}
\item \textbf{Improvement:} Improves the mixing bound from a large function of $\overline{\varphi}, \varphi_{i}$ (\textit{e.g.} roughly $O( \overline{\varphi} \, \max_{1 \leq i \leq n} \varphi_{i})$) to a smaller function (\textit{e.g.} roughly $O(\max (\overline{\varphi}, \, \max_{1 \leq i \leq n} \varphi_{i}))$), at the cost of a poor dependance on the number of parts $n$ of the partition.
\item \textbf{Relevant Results:} Lemma \ref{LemmaSimpSupBdBad2} gives this improvement in the simplest situations; Lemma \ref{LemmaBasicMixing} gives this improvement for more difficult examples. Results in Section \ref{SecDriftBound} may be necessary if one has useful bounds only on $\max_{a \leq i \leq b} \varphi_{i}$ for some $(a,b) \neq (1,n)$. 
\item \textbf{Worked Examples:} Example \ref{SubsecPinceNez} is the simplest example illustrating this improvement. Example \ref{ExKcipToy} obtains a similar improvement in a more complicated situation, using the bounds in Section \ref{SecDriftBound}. 
\item \textbf{Requires:} Bounds on the mixing times $\varphi_{i}$, the expected escape times for $\Omega_{i}$, and lower bounds on the elements $\overline{K}$.
\end{itemize}
\item \textbf{Situation:} As above, but $n$ is large and $K$ exhibits the following \textit{approximate metastability:} there exists some $1 \leq k \leq n$ so that the \textit{mixing time} $\varphi_{i}$ of $K_{i}$ is small compared to the \textit{escape times} $\min_{x \in \Omega_{i}} \E[\min \{ t \, : \, X_{t} \in \cup_{j \in [k] \backslash \{i \}} \Omega_{j} \} \, | \, X_{0} = x\}]$ for all $1 \leq i \leq k$.
\begin{itemize}
\item \textbf{Improvement:} As above.
\item \textbf{Relevant Results:}  The main result is Theorem \ref{ThmContCondHighDim}; as above, results in Section \ref{SecDriftBound} could be helpful.
\item \textbf{Worked Examples:} Example \ref{RemContCondTrace2}.
\item \textbf{Requires:} Bounds on the mixing times $\varphi_{i}$, the expected escape times for $\Omega_{i}$, and a contraction or mixing condition for a projected chain similar to $\overline{K}$.
\end{itemize}
\item \textbf{Situation:} A bound on the mixing time is needed, but $\min_{x} \pi(x)$ is very small.
\begin{itemize}
\item \textbf{Improvement:} Recall that the mixing time $\varphi$ and relaxation time $\varphi^{\mathrm{rel}}$ of a Markov chain satisfy
\be 
\varphi^{\mathrm{rel}} \leq \varphi = O(-\varphi^{\mathrm{rel}} \, \log(\min_{x} \pi(x)));
\ee 
both inequalities are sharp. Thus, bounding the mixing time of a Markov chain directly can give an improvement of a factor of $-\log(\min_{x} \pi(x))$ over a bound that passes through the relaxation time. This factor can be arbitrarily large, and is particularly important for limiting arguments.
\item \textbf{Relevant Results:} All results in this paper can deliver this type of improvement. If the very small value of $\min_{x} \pi(x)$ is the main obstacle to using an existing decomposition bound, we recommend beginning with Section \ref{SecNaiveBoot}, and in particular the relatively weak but simple Inequality \eqref{EqConclusionWellCoveringSimpleExample}. We note that many bounds in Section \ref{SecNaiveBoot} are stated in terms of a somewhat complicated quantity we call the \textit{well-covering number}. As discussed in the introduction to that section, a collection of comparison inequalities allow these bounds to be used without the explicit computation of new well-covering numbers.
\item \textbf{Worked Examples:} Inequality \eqref{EqConclusionWellCoveringSimpleExample}; can be applied to other examples (\textit{e.g.} Section \ref{SubsecPinceNez}).
\item \textbf{Requires:} Inequality \eqref{EqConclusionWellCoveringSimpleExample} requires estimates of the mixing times $\varphi_{i}$ for some values of $i$ as well as entry-wise lower bounds on $\overline{K}$. 
\end{itemize}
\end{itemize}

\section{Main Result: General Mixing Bounds for Decomposable Markov Chains} \label{SecGenMix}
%\subsection{Main Result: Simple Mixing Bound} \label{SecBasicMixing}

Let $\{X_{t} \}_{t \in \mathbb{N}}$ be an irreducible, ${1 \over 2}$-lazy, reversible Markov chain on a finite state space $\Omega$ started at $X_{0} = z$. For $A \subset \Omega$, define 
\be \label{eqn:tauA} 
\tau_{A} = \min \{ t \, : \, X_{t} \in A \} 
\ee
to be the first hitting time of $A$. Theorem 1.1 of \cite{PeSo13} yields that, for $0 < \alpha < \frac{1}{2}$, there exist universal constants $c_{\alpha}$ and $c_{\alpha}'$ such that \footnote{The version of Theorem 1.1 in \cite{PeSo13} and their constants $c_\alpha, c'_\alpha$ are slightly different from ours. See Appendix A for a proof of Equation \eqref{eqn:PS13eqn}.}
\be \label{eqn:PS13eqn}
c'_\alpha \max_{z,A: \pi(A) \geq \alpha} \E_z(\tau_A) \leq \tmix \leq c_\alpha \max_{z,A: \pi(A) \geq \alpha} \E_z(\tau_A).
\ee
Thus \eqref{eqn:PS13eqn} relates mixing times to hitting times of large sets. It is not hard to show $ \tmix \geq c'_\alpha \max_{z,A: \pi(A) \geq \alpha} \E_x(\tau_A)$. The key inequality is the upper bound in \eqref{eqn:PS13eqn}. As discussed in \cite{PeSo13}, this upper bound does not hold for non-reversible chains. For the rest of the paper, $c_\alpha$ and  $c'_\alpha$ will refer to the constants in Equation \eqref{eqn:PS13eqn}.

For $A \subset \Omega$, let $A_{i} = A \cap \Omega_{i}$. Following \eqref{eqn:tauA}, let $\tau_{A_{i}}^{(i)}$ be the first hitting time of $A_{i}$ for the trace $\{ X_{t}^{(i)} \}_{t \in \mathbb{N}}$.  
 The following simple bound on the mixing time of $K$, based on \eqref{eqn:PS13eqn}, forms the basis of our approach:

\begin{lemma} [Basic Mixing Bound] \label{LemmaBasicMixing}
Fix $0 < \alpha < \frac{1}{2}$ and $1 - \alpha < \beta < 1$. Let $0 < \gamma <\min \big({1 \over 2}, \frac{\alpha + \beta - 1}{\beta} \big)$ and fix some $I \subset \{1, 2, \ldots, n \}$ that satisfies
\be 
\sum_{i \in I} \pi(\Omega_{i}) > \beta.
\ee 
Define
\be \label{eqn:Taulem}
 \mathcal{T} \equiv \min \Big\{ T \, : \, \min_{0 < t < T}\max_{i \in I }  \Big( \frac{ \varphi_i }{c'_\gamma t } + \max_{z \in \Omega} \P_{z}[\kappa_{i}(T) < t] \Big) < \frac{1}{4} \Big\}.
\ee 
Then the mixing time $\tmix$ of $\{ X_{t} \}_{t \in \mathbb{N}}$ satisfies
\be 
\tmix \leq \frac{4}{3} c_\alpha \mathcal{T}.
\ee 
\end{lemma}
\begin{proof}
Fix a set $A \subset \Omega$ with $\pi(A) \geq \alpha$. We will denote $A_i = A \cap \Omega_i$ for all $1 \leq i \leq n$. We claim that there exists $j \in I$ so that 
\be \label{IneqBigToSmallSets}
\pi_{j}(A_j) \geq \gamma > 0.
\ee 
To see this, assume that inequality \eqref{IneqBigToSmallSets} is not satisfied for any $i \in I$. Then we would have 
\be 
\pi(A) &= \sum_{i \notin I} \pi(A_i) + \sum_{i \in I} \pi(A_i) \\
&< (1 - \sum_{i \in I} \pi(\Omega_{i})) + \gamma \sum_{i \in I} \pi(\Omega_{i})\\
&= 1 - (1 - \gamma)  \sum_{i \in I} \pi(\Omega_{i})\\
&\leq 1 - (1 - \gamma)\beta \leq \alpha,
\ee  
contradicting the assumption that $\pi(A) \geq \alpha$. Thus, inequality \eqref{IneqBigToSmallSets} is satisfied. \par 
Let $j \in I$ be an index satisfying $\pi_{j}(A_j) \geq \gamma$. For any $T \in \mathbb{N}$ and  $0 < t < T$,  we have 
\be
\{\tau_A < T\} \supset \{\tau_{A_j} < T\} = \{\tau_{A_j}^{(j)} < \kappa_j(T)\} \supset \{\tau^{(j)}_{A_j} < t\} \cap \{\kappa_j(T) > t\}.
\ee
This gives
\be 
\P_{z}[\tau_{A} > T] &\leq \P_z[\tau_{A_{j}} > \kappa_{j}(T)] \\
&\leq \max_{y \in \Omega_{j}} \P_{y}[\tau_{A_{j}}^{(j)} > t] +  \P_{z}[\kappa_{j}(T) < t ] \\
&\leq \max_{y \in \Omega_{j}} \frac{\E_{y}[\tau_{A_{j}}^{(j)}]}{t} + \P_{z}[\kappa_{j}(T) < t ] \\ 
&\leq  \frac{\varphi_j}{ c_{\gamma}' t} + \P_{z}[\kappa_{j}(T) < t ], 
\ee  
where the last inequality follows from Equation \eqref{eqn:PS13eqn}. Thus, for $T \geq \mathcal{T}$,  
\be 
\max_{z \in \Omega} \P_{z}[\tau_{A} > T] \leq \frac{1}{4}.
\ee 
Since this holds for all $z \in \Omega$ and all $A$ with $\pi(A) > \alpha$, by Equation \eqref{eqn:subexp} we have 
\be
\max_{z \in \Omega, \, \pi(A) \geq \alpha} \E_{z}[\tau_{A}] \leq \frac{4}{3} \mathcal{T}.
\ee 
By Equation \eqref{eqn:PS13eqn}, 
\be 
\tmix \leq c_\alpha \max_{z \in \Omega, \, \pi(A) \geq \alpha} \E_{z}[\tau_{A}] \leq  \frac{4}{3} c_{\alpha} \mathcal{T}, 
\ee  
completing the proof.
\end{proof}

Our next result weakens the requirement in Lemma \ref{LemmaBasicMixing} that $\max_{z \in \Omega} \P_{z}[\kappa_{i}(T) < t]$ must be small for all $i \in I$ to the requirement that only $\max_{z \in \Omega} \P_{z}[\cap_{i \in I} \{\kappa_{i}(T) < t\}]$ needs to be small, at the cost of requiring the bound be uniform over all $I \subset [n]$ with $\pi(\cup_{i \in I} \Omega_{i})$ sufficiently large. Define
\be 
 \mathcal{T}' = \min \Big\{ T \, : \, \min_{0 < t < T} \Big(\max_{I \subset [n] : \pi(\cup_{i \in I} \Omega_{i}) \geq \frac{\alpha}{2} } \big( \max_{z \in \Omega} \P_{z}[\cap_{i \in I} \{\kappa_{i}(T) < t\}] + \sum_{i \in I} e^{- \lfloor \frac{ c_{\frac{\alpha}{2}}' \, t }{e \, \varphi_{i}} \rfloor} \big) \Big) < \frac{1}{4} \Big\}. \\
 \label{eqn:taup}
\ee

In the proof of the following lemma, we bound the distribution of the number of steps in an excursion from a given set $A$ by the distribution of the hitting time to $A$ from the worst possible starting point in $\Omega$:

\begin{lemma} \label{LemmaBasicMixing2}
Fix $0 < \alpha < \frac{1}{2}$ and let $\mathcal{T}'$ be as in \eqref{eqn:taup}. Then the mixing time $\tmix$ of $\{ X_{t} \}_{t \in \mathbb{N}}$ satisfies
\be 
\tmix \leq   \frac{4}{3} c_\alpha \mathcal{T}'.
\ee 
\end{lemma}
\begin{proof}
The proof is similar to that of Lemma \ref{LemmaBasicMixing}.
Fix a set $A \subset \Omega$ with $\pi(A) \geq \alpha$. 
Define the set  
\be 
J_{A} = \{i \in [n] \, : \,\frac{\pi(\Omega_{i} \cap A)}{\pi(\Omega_{i})} \geq \frac{\alpha}{2} \}.
\ee 
We claim that
\be \label{IneqSimpLemmCoverage}
\pi(\cup_{i \in J_{A}} \Omega_{i}) \geq \frac{\alpha}{2 - \alpha} \geq \frac{\alpha}{2}.
\ee \label{IneqSimpLemmGoodInt2}

To see this, assume that $\pi(\cup_{i \in J_{A}} \Omega_{i}) = p < {\alpha \over 2 - \alpha}$. Then
\be 
\alpha &\leq \pi(A) \\
&= \sum_{i} \pi(\Omega_{i} \cap A) \\
&= \sum_{i \in J_{A}} \pi(\Omega_{i}) \frac{\pi(\Omega_{i} \cap A)}{\pi(\Omega_{i})} + \sum_{i \notin J_{A}} \pi(\Omega_{i}) \frac{\pi(\Omega_{i} \cap A)}{\pi(\Omega_{i})} \\
&< \sum_{i \in J_{A}} \pi(\Omega_{i})  + \sum_{i \notin J_{A}} \pi(\Omega_{i}) \frac{\alpha}{2} \\
&= p + \frac{\alpha}{2} \big(1 -p \big) < \alpha, \label{eqn:pa}
\ee 
yielding a contradiction. The final inequality of \eqref{eqn:pa} follows from the fact that 
$ p + \frac{\alpha}{2} \big(1 -p \big) <\alpha$ for $p < {\alpha \over 2-\alpha}$.
Thus \eqref{IneqSimpLemmGoodInt2} holds. \par
 The goal is to bound $\tau_{A}$ by $\{ \tau_{A_{i}}^{(i)} \}_{i \in J_{A}}$. For all $T \in \mathbb{N}$ and $z \in \Omega$, 
\be 
\P_{z}[\tau_{A} > T] &= \P_{z}\big[\cap_{i \in J_{A}} \{ \tau_{A_{i}}^{(i)} > \kappa_{i}(T)\}\big] \\
&\leq \min_{0 \leq t \leq T} \P_{z} \big[\cap_{i \in J_{A}} ( \{ \tau_{A_{i}}^{(i)} > t \} \cup \{ \kappa_{i}(T) < t \} )\big]  \\
&\leq \min_{0 \leq t \leq T} (\max_{x \in \Omega} \P_{x}\big[\cap_{i \in J_{A}}\{\kappa_{i}(T) < t\}\big] +  \P_{x}[ \cup_{i \in J_{A}} \{\tau_{A_{i}}^{(i)} > t\}]) \\
&\leq \min_{0 \leq t \leq T} (\max_{x \in \Omega}\P_{x}\big[\cap_{i \in J_{A}}\{\kappa_{i}(T) < t\}\big]+ \sum_{i \in J_{A}} e^{- \big\lfloor \frac{t}{e \, \max_{y \in \Omega_{i}} \E_{y} [\tau_{A_{i}}^{(i)}]} \big\rfloor}  ) \\
&\leq \min_{0 \leq t \leq T} (\max_{x \in \Omega} \P_{x}\big[\cap_{i \in J_{A}}\{\kappa_{i}(T) < t\}\big]+ \sum_{i \in J_{A}} e^{- \lfloor \frac{ c_{\frac{\alpha}{2}}'\, t   }{ e \,  \varphi_{i}} \rfloor}),
\ee 
where the penultimate inequality follows from inequality \eqref{eqn:subexp2} and the last inequality follows from inequality \eqref{eqn:PS13eqn}. Since $\pi(\cup_{i \in J_{A}} \Omega_{i}) \geq \frac{\alpha}{2}$, for $T > \mathcal{T}'$, we have 
\be 
\max_{z \in \Omega} \P_{z}[\tau_{A} > T] \leq \frac{1}{4}.
\ee 
Since this holds for  all $A$ with $\pi(A) > \alpha$, we have by equation \eqref{eqn:subexp} 
\be
\max_{z \in \Omega, \, \pi(A) \geq \alpha} \E_{z}[\tau_{A}] \leq \frac{4}{3} \mathcal{T}.
\ee 
By Equation \eqref{eqn:PS13eqn}, 
\be 
\tmix \leq c_\alpha \max_{z \in \Omega, \, \pi(A) \geq \alpha} \E_{z}[\tau_{A}] \leq   \frac{4}{3} c_{\alpha} \mathcal{T}', 
\ee  
completing the proof.
\end{proof}

The following example shows that Lemma \ref{LemmaBasicMixing2} can be much stronger than Lemma \ref{LemmaBasicMixing}:

\begin{example} \label{ExLemmaCompStatement}
Fix $3 \leq d \in \mathbb{N}$, let $\{ G_{m} \}_{m \geq d+1}$ be a sequence of $d$-regular expander graphs with $| G_{m} | = m$ (see, \textit{e.g.,} Theorem 4.16 of \cite{hoory2006expander} for proof that such a sequence exists). Let $Q_{m}$ be the kernel associated with $\frac{3}{4}$-lazy simple random walk on $G_{m}$. By Theorem 3.2 of \cite{hoory2006expander}, the mixing time of $Q_{m}$ is $O(\log(m))$. Fix a sequence $\epsilon = \epsilon_{m} < \min(\frac{1}{4},\frac{1}{\log(m)})$, let $H_{m} = \{ (i,v) \, : \, i \in \{1,2\}, v \in G_{m} \}$ and define the kernel $K_{m}$ by setting
\be 
K_{m}( (1,u), (1,v) ) = Q_{m}(u,v), \, \, K_{m}((1,u),(2,u)) = \frac{1}{2}, \, \, K_{m}((2,u),(1,u)) = \epsilon,
\ee 
setting $K_{m}(x,y) = 0$ for all $y \neq x$ not of the form listed above, and finally setting $K_{m}(x,x) = 1 - \sum_{y \neq x} K_{m}(x,y) \geq \frac{1}{4}$. Let $\pi$ denote the stationary measure of $K_m$ and $\tmix$ denote its mixing time.

For each $m \in \mathbb{N}$, we consider the partition $\Omega = \sqcup_{u \in G_{m}} \Omega_{u}$ of $\Omega = H_{m}$ into the $m$ sets $\Omega_{u} = \{ (1,u), (2,u) \}$. We will show that Lemma \ref{LemmaBasicMixing} cannot obtain any bound on the mixing time stronger than $\tmix = O(m)$, while Lemma \ref{LemmaBasicMixing2} can be used to obtain the correct bound of $\tmix = O(\epsilon^{-1} \log(m))$. When $\log(m) \leq \epsilon^{-1} \ll \frac{m}{\log(m)}$, this is a substantial difference. 

By symmetry, $\pi(\Omega_u) = {1 \over m}$ for all $u \in G_m$. Fix some $\alpha = \frac{1}{3}$. Then the set $I$ used in the statement of Lemma \ref{LemmaBasicMixing} must be of size at least $\frac{m}{6}$. Since the restriction of $K_{m}$ to $\Omega_{u}$ has only two points, the mixing time $\varphi_{u}$ can be computed explicitly; it is  $\Theta(1)$. Lemma \ref{LemmaBasicMixing} requires that for some $t< \mathcal{T}$ (see Equation \eqref{eqn:Taulem}), both $\frac{\varphi_{\max}}{ c_{\frac{\alpha}{2}}' \, t} < {1\over 4}$ and
$\P[\kappa_{i}(\mathcal{T}) < t] < {1 \over 4}$ for all $i \in I$. If 
$\P[\kappa_{i}(\mathcal{T}) < t] < {1\over 4}$ holds for all $i \in I$, the pigeonhole principle implies that we must have $\mathcal{T} \geq C |I| \,t \geq C' m t$ for some universal constants $C,C' >0$. Thus, applying Lemma \ref{LemmaBasicMixing} cannot yield a better upper bound than $\tmix = O( \mathcal{T})$ where $m = O(\mathcal{T})$.\par
We briefly sketch an argument showing that is possible to obtain a much better upper bound via Lemma \ref{LemmaBasicMixing2}. Again, set $\alpha = \frac{1}{3}$ and fix $I \subset G_{m}$ with $| I | \geq \frac{\alpha}{2} m = {m \over 6}$. Let $\{X_{t} \}_{t \in \mathbb{N}}$ be a walk evolving on $\Omega$ according to $K_{m}$, define $\Omega_{\mathrm{lower}} = \{ (1,u) \, : \, u \in G_{m} \}$. Let $\{Y_{t} \}_{t \in \mathbb{N}}$ be the trace of $\{ X_{t} \}_{t \in \mathbb{N}}$ on $\Omega_{\mathrm{lower}}$. Identifying elements of $\Omega_{\mathrm{lower}}$ with points of $G_{m}$ by the map $(1,u) \mapsto u$, it can be verified that $\{Y_{t} \}_{t \in \mathbb{N}}$ is a Markov chain with transition kernel $Q_{m}$. Since the mixing time of $Q_{m}$ is $O(\log(m))$, Equation \eqref{eqn:PS13eqn} implies that the expected hitting time for $\{ Y_{t} \}_{t \geq 0}$ of any subset $I \subset G_{m}$ of size $|I| \geq \frac{m}{6}$ is $O(\log(m))$, uniformly in the starting vertex $Y_{0}$.  Let 
\be
\tau_{I}' = \min \{t > 0 \, : \, Y_{t} \in \cup_{u \in I} \Omega_{u} \}.
\ee
 As noted earlier, the mixing time of $Q_m$ is $O(\log(m))$. Thus by \eqref{eqn:PS13eqn}, there exists some constant $0 < C< \infty$ that does not depend on $Y_{0}=y$, $m$ or the particular set $I$  so that 
\be \label{IneqCompExampHittingPrime}
\E_{y}[\tau_{I}'] \leq C \log(m).
\ee  Next, let 
\be
 \tau_{I} = \min \{ t > 0 \, : \, X_{t}  \in \cup_{u \in I} \Omega_{u} \cap \Omega_{\mathrm{lower}} \}.
\ee
Set $\tilde{\eta}(0) = -1$, inductively define $\tilde{\eta}(s+1) = \min \{t > \tilde{\eta}(s) \, : \, X_{t} \in \Omega_{\mathrm{lower}} \}$, and set $\tilde{\kappa}(s) = \max \{ u \, : \, \tilde{\eta}(u) \leq s \}$. We then have 
\be \label{ExLemmCompSimp}
\tau_{I}' = \tilde{\kappa}(\tau_{I}).
\ee 
Next, $\{ \tilde{\eta}(s+1) - \tilde{\eta}(s)\}_{s \in \mathbb{N}}$ is an i.i.d. sequence with mean $O(\epsilon^{-1})$. Thus, there exists a constant $C_1>0$ so that
\be \label{ExLemmCompExMark2}
\max_{x \in \Omega} \P_{x}[\tilde{\kappa}(C_{1} A_{1} \log(m) \epsilon^{-1}) <  A_{1} \log(m)] \leq {1 \over 8}
\ee  
for all $A_{1} > 0$. Combining inequalities \eqref{IneqCompExampHittingPrime}, \eqref{ExLemmCompSimp} and \eqref{ExLemmCompExMark2} gives 
\be \label{ExLemmCompExMarkConc}
\tau_{I}  = O( \log(m) \epsilon^{-1}),
\ee 
uniformly in $X_{0} = x \in \Omega$.

Using again the observation that $\tilde{\eta}(s+1) - \tilde{\eta}(s)$ is a sequence of i.i.d. random variables, each a sum of geometric random variables and with mean that is $O(\epsilon^{-1})$, we have for all $A_{2} > 0$, 
\be \label{IneqRegExInferHa}
\min_{x \in \Omega} \P_{x}[ \{X_{t}\}_{t = \tau_{I} + 1}^{\tau_{I} + A_{2} \log(m)} \subset \Omega_{\mathrm{lower}}^{c}] \geq C (1-\epsilon)^{C' A_{2} \log(m)} \geq C'' e^{-A_{2}} 
\ee 
for some $0 < C, C', C'' < \infty$ that do not depend on $\epsilon$ or $m$.
Recall that $\kappa_i(\cdot)$ denotes on the occupation time of $X_t$ on $\Omega_i$.
Since $\varphi_{\max} = \Theta(1)$, combining inequalities \eqref{IneqRegExInferHa} and \eqref{ExLemmCompExMarkConc}, we infer that there exists a  constant $C_2$ such that 
\be  \label{eqn:supexlem2all}
\max_{z \in \Omega} \P_{z}[\cap_{i \in I}\{\kappa_{i}(C_{2} \log(m) \epsilon^{-1}) < \lceil \frac{8}{c_{\frac{1}{6}}'}\varphi_{\max} \log(m) \rceil \}] <   {1 \over 8}.
\ee 
By choosing $t =  \lceil \frac{8}{c_{\frac{1}{6}}'}\varphi_{\max} \log(m) \rceil$ and $T = C_{2} \log(m) \epsilon^{-1}$ with $C_2$ sufficiently large, we obtain
\be 
\max_{I \subset G_{m} : \pi(\cup_{i \in I} \Omega_{i}) \geq \frac{1}{6} } \sum_{i \in I} e^{-\lfloor \frac{c_{\frac{1}{6}}' t}{ \varphi_i} \rfloor} + \max_{z \in \Omega} \P_{z}[\cap_{i \in I} \{ \kappa_{i}(T) < t\}]  < {1 \over 8} + {1 \over 8} = \frac{1}{4}.
\ee
From Equation \eqref{eqn:taup} and Lemma \ref{LemmaBasicMixing2}, it follows that $\tmix = O(\log(m) \epsilon^{-1})$, which indeed is the correct mixing time. For $\epsilon^{-1} \ll m$, this is much better than any bound obtainable by applying Lemma \ref{LemmaBasicMixing}.
\end{example}

\begin{remark}
The partition used in Example \ref{ExLemmaCompStatement} 
was $\Omega_u = \{(1,u),(2,u)\}, u \in G_m$. A more natural partition is
$H_m =  \Omega_{\mathrm{lower}} \cup \Omega_{\mathrm{lower}}^{c}$.  Lemma \ref{LemmaBasicMixing} applied to this partition indeed gives the optimal bound of $O(\log(m) \epsilon^{-1})$ for the mixing time of $K_m$, without the restriction that $\epsilon < \frac{1}{\log(m)}$. However, our main point behind Example \ref{ExLemmaCompStatement} is not about \emph{choosing} partitions, but to illustrate that for a \emph{given} partition, Lemmas \ref{LemmaBasicMixing} and \ref{LemmaBasicMixing2} can give completely different answers. 
\end{remark}

Fix $0 < \alpha < \frac{1}{2}$ and define 
\be \label{EqDefAvgHitTime}
\overline{\varphi}_{\mathrm{hit}} = \max_{I \subset [n] \ : \,  \pi(\cup_{i \in I} \Omega_{i}) \geq \frac{\alpha}{2}} \max_{x \in \Omega} \E_{x}[\tau_{\cup_{i \in I} \Omega_{i}}].
\ee

We give a corollary to Lemma  \ref{LemmaBasicMixing2}.
Fix $i \in [n]$ and $X_{0} = x \in \Omega_{i}$. Define the escape time 
\be \label{eqn:escapetime}
\tEsc[i] = \min \{ t > 0 \, : \, X_{t} \notin \Omega_{i} \}.
\ee

\begin{cor} \label{CorSimpSupBdBad2}
Assume that, for some $\epsilon, \delta > 0$,
\be \label{IneqInitialRegularityCondition}
\min_{i \in [n]} \min_{x \in \Omega_{i}} \P_{x}[\tEsc[i] > \epsilon \varphi_{\max}] \geq \delta.
\ee
Following the notation of Lemma  \ref{LemmaBasicMixing2} and Equation \eqref{EqDefAvgHitTime}, we have
\be 
\tmix = O (  \epsilon^{-1} \delta^{-1} \overline{\varphi}_{\mathrm{hit}} n \log(n)).
\ee 
\end{cor}
\begin{proof}
Let $I \subset[n]$, $A = \cup_{i \in I} \Omega_{i}$ satisfy $\pi(A) > \frac{\alpha}{2}$ and let $\kappa_{I}(s)$, $\eta_{I}(s)$ be defined as in Equation \eqref{EqDefOccupationMeasureCounters} with $\Omega_{i}$ replaced by $\cup_{i \in I} \Omega_{i}$. Let $\{ Z_{i} \}_{i \in \mathbb{N}}$ be a sequence of i.i.d. geometric variables with $\P[Z_{1} > 1] = e^{-1}$ and let $\{ Z_{i}' \}_{ i \in \mathbb{N}}$ be a sequence of i.i.d. Bernoulli random variables with $\P[Z_{1}'=1] = 1 - \P[Z_{1}'=0] = \delta$.

We make two observations. The random variable $e \varphi_{\max} \, Z_{j}$ stochastically dominates the return times $(\eta_{I}(j) - \eta_{I}(j-1))$ conditional on $X_{\eta_{I}(j-1)}$. The random variable $Z_{j}'$ is stochastically dominated by the random variable $\textbf{1}_{\mathcal{E}(j)}$ conditional on $X_{\eta_{I}(j)}$, where 
\be
\mathcal{E}(j) = \big \{ \cup_{s= \eta_{I}(j)}^{\eta_{I}(j) + \lceil \epsilon \varphi_{\max} \rceil} \{ X_{s} \} \in \cup_{i \in I} \Omega_{i} \big \}.
\ee
Thus $\mathcal{E}$ denotes the event that the $j$'th visit to  $\cup_{i \in I} \Omega_{i}$ to be of length at least $\epsilon \varphi_{\max}$. These two observations give, for any $0 \leq t \leq T \in \mathbb{N}$ and $S \geq 1$, 
\be \label{IneqCorrRegMainBound}
\P\Big(\max_{i \in I} \kappa_{i}(T + \epsilon \varphi_{\max}) &< \frac{t}{|I|}\Big) \leq \P[\kappa_{I}(T+ \epsilon \varphi_{\max}) < t] \\
&\leq \P[  \sum_{i=1}^{S} (\eta_{I}(j) - \eta_{I}(j-1)) > T] + \P[\sum_{i=1}^{S} \textbf{1}_{\mathcal{E}(j)} \leq \frac{t}{\epsilon \varphi_{\max}}]  \\
&\leq \P[ e \, \overline{\varphi}_{\mathrm{hit}} \sum_{i=1}^{S} Z_{i} > T] + \P[\sum_{i=1}^{S} Z_{i}' \leq \frac{t}{\epsilon \varphi_{\max}}],
\ee  
where the last inequality follows from Equation \eqref{eqn:subexp2} and Equation \eqref{IneqInitialRegularityCondition}.

Fix $0 < C_{1} < \infty$ and set $t = C_{1} \varphi_{\max} n \log(n)$. For any such choice of $C_{1}$, there exists a $0 < C_{2} < \infty$ so that for $S > C_{2} C_{1} \epsilon^{-1} \delta^{-1} n \log(n)$,
\be 
\P[\sum_{i=1}^{S} Z_{i}' \leq \frac{t}{\epsilon \varphi_{\max}}] \leq \frac{1}{16},
\ee 
and for any such choice of $0 < C_{1},C_{2} < \infty$, there exists a $0 < C_{3} < \infty$ so that for $T > C_{1} C_{2} C_{3}  \epsilon^{-1} \delta^{-1} \overline{\varphi}_{\mathrm{hit}} n \log(n)$,
\be 
 \P[ e \, \overline{\varphi}_{\mathrm{hit}} \sum_{i=1}^{S} Z_{i} > T]\leq \frac{1}{16}.
\ee 
Combining these two bounds with inequality \eqref{IneqCorrRegMainBound} gives, for these choices of $t,T$,
\be 
\P[\max_{i \in I} \kappa_{i}(T) < C_{1} \varphi_{\max} \log(n)] \leq \frac{1}{8}.
\ee 

The result now follows from Lemma \ref{LemmaBasicMixing2}. 

\end{proof}

\begin{remark}
Taking $\epsilon^{-1} = \varphi_{\max}$ and $\delta = 1$ in Corollary \ref{CorSimpSupBdBad2} gives 
\be \label{IneqVisSimilar}
\tmix = O ( \overline{\varphi}_{\mathrm{hit}} \varphi_{\max} n \log(n)),
\ee 
which is visually similar to the main decomposition bounds in \cite{JSTV04} and other papers cited in the Introduction.
\end{remark}

The following bound on $ \overline{\varphi}_{\mathrm{hit}}$ gives an easy way to use Corollary \ref{CorSimpSupBdBad2} when $n$ is small:

\begin{lemma} \label{LemmaSimpSupBdBad2}
Follow the notation of Corollary  \ref{CorSimpSupBdBad2}. Assume also 
\be \label{IneqFurtherRegularityConditions}
\max_{i \in [n]} \max_{x \in \Omega_{i}}  \P_{x}[\tEsc[i] > \epsilon \varphi_{\max}] \leq 1 - \delta
\ee 
for some $\epsilon, \delta > 0$.
For $c > 0$, let $G_{c}$ be the directed graph with vertices $V_{c} = \Omega$ and edges 
\be 
E_{c} = \{ (i,j) \, : \, \min_{x \in \Omega_{i}} \P_{x}[X_{\tEsc[i]} \in \Omega_{j}] \geq c \}.
\ee 
Let $D$ be the diameter of $G_{c}$. Then 
\be 
\overline{\varphi}_{\mathrm{hit}} \leq  \epsilon \varphi_{\max} D (c \delta)^{-D}.
\ee 
\end{lemma}
\begin{proof}
Fix $I \subset [n]$ satisfying $ \pi( \cup_{i \in I}\Omega_{i}) \geq \frac{\alpha}{2}$ and let $A = \cup_{i \in I}\Omega_{i}$. Equation \eqref{IneqFurtherRegularityConditions} and the definition of $G_{c}$ immediately give
\be 
\min_{x \in \Omega} \P_{x}[ \tau_{A} \leq \epsilon \varphi_{\max} D] \geq (c \delta)^{D}.
\ee 
By Equation \eqref{eqn:subexp}, this implies
\be 
\max_{x \in \Omega} \E_{x}[\tau_{A}] &= \max_{x \in \Omega}  \sum_{t =0}^{\infty} \P_{x}[\tau_{A} > t] \leq  \max_{x \in \Omega}   \sum_{k=0}^{\infty} \epsilon \varphi_{\max} D \P_{x}[ \tau_{A} > k \epsilon \varphi_{\max} D] \\
&\leq \epsilon \varphi_{\max} D    \sum_{k=0}^{\infty}  (\max_{x \in \Omega} \P_{x}[ \tau_{A} >  \epsilon \varphi_{\max} D])^{k} \leq \epsilon \varphi_{\max} D    \sum_{k=0}^{\infty}  (1 - (c \delta)^{D} )^{k} =  \epsilon \varphi_{\max} D (c \delta)^{-D}.
\ee 
Since this holds for all sets $I \subset [n]$ satisfying $ \pi( \cup_{i \in I}\Omega_{i}) \geq \frac{\alpha}{2}$, this completes the proof.
\end{proof}

Lemmas \ref{LemmaBasicMixing} and \ref{LemmaBasicMixing2} allow us to bound the mixing time  $\tmix$  of $\{X_{t} \}_{t \in \mathbb{N}}$ in terms of the mixing times $\{ \varphi_{i} \}_{i=1}^{n}$ of the traces of $\{X_{t} \}_{t \in \mathbb{N}}$ as well as the occupation times $\{ \kappa_{i} \}_{i=1}^{n}$. The assumption that we have good bounds on $\varphi_{i}$ is similar to the assumption of a good bound on the relaxation times of restricted chains in \cite{JSTV04}, and it often holds. However, the occupation times are much more difficult to understand than the relaxation times of the projected chain used in \cite{JSTV04}. The remainder of this paper is devoted to building tools for bounding these occupation times. 

\section{Mixing Bounds Via Well-Covering Times} \label{SecNaiveBoot}
The mixing bounds obtained in this section are based on the following observations:

\begin{enumerate}
\item If the occupation measure $\kappa_{i}(T)$ is large relative to $\varphi_{i}$ for some set $\Omega_{i}$, with high probability it will also be large for `neighboring' states $\Omega_{j}$ with $\overline{K}(i,j)$ large.
\item If the high-probability event described above holds for all pairs $i,j$, then every set $\Omega_{i}$ with large stationary measure will also have large occupation measure.
\end{enumerate}

We begin our development by defining a quantity that we call the \textit{well-covering time} that allows us to make this observation more precise. 

Readers interested in quickly obtaining reasonable bounds may skip this definition on a first reading of the paper: the comparison theory for well-covering times developed in Section \ref{SubsecCompWellCov}, combined with the bound on the well-covering time of a simple example computed in Section \ref{SubsecWellCovEx}, allows users to estimate well-covering times without computing any directly. Thus, the main result of this section (Theorem \ref{ThmMainBootstrapBound}) can be used in simple examples without using this definition. For more complicated examples, we still suggest estimating the well-covering times directly only for simple examples as in Section \ref{SubsecWellCovEx}, then using the comparison bounds to relate these to the Markov chain of interest.

\subsection{Well-Covering Times}
Fix $n \in \mathbb{N}$ and let $Q$ be a reversible, irreducible, aperiodic  transition kernel on $[n]$ with stationary measure $\mu$. Fix constants $0 < B, T, t_{1},\ldots,t_{n} <\infty$ and define the set $\mathcal{S}_{B,T}$ to be the pairs $\kappa = (\kappa(1),\cdots,\kappa(n)) \in [0,1]^{n}$, $N = [N(i,j)] \in [0,1]^{n} \times [0,1]^{n}$ satisfying $\sum_{j} \kappa(j) = 1$, $\sum_{i,j} N(i,j) = 1$ and the inequalities:
\be [ineqWellCovDef]
|N(i,j) - \kappa(i) Q(i,j) |, \, |N(i,j) - \kappa(j) Q(j,i) |  &\leq \frac{B \sqrt{\kappa(i)}}{\sqrt{T}},  \\
| \sum_{i} N(i,j) - \kappa(j) |, \, | \sum_{j} N(i,j) - \kappa(i) |  &\leq \frac{1}{ T}.
\ee 
The terms $\kappa$ and $N$ in \eqref{ineqWellCovDef} represent rescaled counts of the empirical occupation measures and transition counts for a Markov chain with transition kernel $K$ and projected kernel $\overline{K} = Q$. For fixed constant $B$, the set $\mathcal{S}_{B,T}$ represents all \textit{plausible} joint values of $\kappa,N$ over a run of the Markov chain for $T$ steps. For example, while it is \textit{possible} to have $\kappa$ concentrated on one part $\Omega_{1}$ of the partition of a Markov chain for a time $T \gg \varphi_{1} \sum_{j \neq 1} \overline{K}(1,j)$, this event is extremely unlikely and so for fixed $B$, the pair $\kappa = (1,0,\ldots)$, $N \equiv 0$ will not be in $\mathcal{S}_{B,T}$ for  $T$ large. \par
$\mathcal{S}_{B,T}$ is of interest for $T$ large. In the limit $T=\infty$, the set of Equations \eqref{ineqWellCovDef} yield
\be\label{eqn:Tinf1}
N(i,j) = \kappa(i) Q(i,j) = \kappa(j) Q(j,i)
\ee
and
\be \label{eqn:Tinf2}
\sum_{i}N(i,j) = \kappa(j),  \sum_{j}N(i,j) = \kappa(i).
\ee
Summing \eqref{eqn:Tinf1} with respect to the index $i$ and using \eqref{eqn:Tinf2} gives
\be
\sum_{i} \kappa(i) Q(i,j)&=\sum_i N(i,j) 
= k(j).
\ee
Thus the probability vector $\kappa$ satisfies $\kappa = \kappa Q$. Since $Q$ is irreducible and aperiodic, it follows that $\kappa = \mu$ when $T = \infty$.

\begin{defn} [Well-Covering Time]
The \textit{well-covering} time $\tau_{\mathrm{wc}} = \tau_{\mathrm{wc}} (t_{1},\ldots,t_{n},  B)$ associated with the kernel $Q$ and the associated set $\mathcal{S}_{B,T}$ satisfying \eqref{ineqWellCovDef} is defined as
\be 
\tau_{\mathrm{wc}} = \min \left \{ T > 0 \, : \, \forall \, (\kappa, N) \in \mathcal{S}_{B,T}, \, \forall \, i \in [n], \, \kappa(i) > \frac{t_{i}}{T} \right  \}.
\ee 
\end{defn}
This definition might seem slightly unwieldy at first glance. We give more intuition on this quantity in the following sections. 
\subsection{Well-Covering Times for the simple Random Walk on a Graph} \label{SubsecWellCovEx}
We compute the well-covering time of a simple random walk on a graph. The calculations in this section can be combined with the comparison results in Section \ref{SubsecCompWellCov} to obtain crude but useful bounds on the well-covering numbers of a much broader collection of Markov chains.\par Let $G = (V,E)$ be a tree with $|V| = n$ vertices, maximum degree less than $\Delta$ and diameter $D$. Let $Q$ be the transition kernel on state space $[n]$ given by
\be 
Q(i,j) = \frac{1}{2 \Delta} \textbf{1}_{(i,j) \in E}
\ee 
for $i \neq j$ and $Q(i,i) = 1 - \sum_{j \neq i} Q(i,j)$.

This is a kernel associated with simple random walk on $G$, and its stationary distribution is $\mu(i) = \frac{1}{n}$ for all $i \in V$. 
\begin{lemma} \label{LemmaWellCovEx}Let $\tau_{\mathrm{wc}}$ be the well-covering time associated with the kernel $Q$. For any $\phi > 0$, 
\be \label{IneqFirstWellCoveringBound}
\tau_{\mathrm{wc}}(\phi, \ldots, \phi, B) \leq n \max( 10^3 \Delta^{2} B^{2} D^{2}, 4 \phi).
\ee 
\end{lemma}
\begin{proof}
Let $T \geq n \max( 10^3 \Delta^{2} B^{2} D^{2}, 4 \phi)$ and let $(\kappa, N) \in \mathcal{S}_{B,T}$. By the pigeonhole principle, there exist a vertex $i$ such that $\kappa(i) \geq \frac{1}{n}$. For $u,v \in G$, denote by $|u-v|$ the graph distance on $G$, that is, the length of the shortest path in $G$ from $u$ to $v$. By induction on the quantity $|i-j|, j \in V$, we will show that 

\be \label{IneqIndHypCircle}
\kappa(j) \geq \kappa(i) e^{-D \frac{16B \Delta \sqrt{n}}{\sqrt{T}}} \geq \frac{1}{4n}.
\ee 
It is clear that inequality \eqref{IneqIndHypCircle} holds for $|i-j| = 0$. To prove the inequality for all $j$, fix $0 < s \leq D$  and assume that inequality \eqref{IneqIndHypCircle} holds for all $j$ such that $|i-j| < s$; we will prove that it holds for $j$ such that $|i-j| = s$. Fix $j$ so that $|i-j| = s$ and also fix $\ell \in G$ that satisfies $|i-\ell| = s-1$ and $|\ell-j| = 1$; by the definition of the graph distance, at least one such vertex exists. By inequality \eqref{ineqWellCovDef}, 
\be 
N(\ell,j) &\geq \frac{1}{2 \Delta} \kappa(\ell) - \frac{B \sqrt{\kappa(\ell)}}{\sqrt{T}} \\
\kappa(j) \frac{1}{2 \Delta} &\geq N(\ell,j) -  \frac{B \sqrt{\kappa(\ell)}}{\sqrt{T}}. \\
\ee 
Combining these two inequalities,
\be \label{IneqWellCovBasicBit}
\kappa(j) \geq \kappa(\ell) -  \frac{4 B \Delta \sqrt{\kappa(\ell)}}{\sqrt{T}} = \kappa(\ell)\big(1 - \frac{4 B \Delta }{\sqrt{\kappa(\ell)} \sqrt{T}}\big).
\ee 
By the induction hypothesis \eqref{IneqIndHypCircle}, 
\be 
\kappa(j) \geq \kappa(\ell) (1 - \frac{16B \Delta \sqrt{n}}{\sqrt{T}}) \geq \kappa(i) e^{-s \frac{16B \Delta \sqrt{n}}{\sqrt{T}}}.
\ee 
Since $T \geq  10^3 \Delta^{2} B^{2}  D^{2} n$ and $j \leq D$, this implies 
\be 
\kappa(j) \geq \frac{1}{4n},
\ee 
proving inequality \eqref{IneqIndHypCircle} for the case $|i-j| = s$ and thus completing the induction argument.  Since $T \geq 4 n \phi$, inequality \eqref{IneqIndHypCircle} implies
\be 
\kappa(j) \geq \frac{\phi}{T} 
\ee 
for all $j \in [n]$, completing the proof of inequality \eqref{IneqFirstWellCoveringBound}. 
\end{proof}
\subsection{Well-Covering Time Bounds Mixing Times}
For $1 \leq i,j \leq n$, define the first $\Omega_{i} \rightarrow \Omega_{j}$ transition time $T_{i,j}^{(0)} = 0$, then define subsequent transition times by 
\be 
T_{i,j}^{(t+1)} = \min \{ s > T_{i,j}^{(t)} \, : \, X_{s} \in \Omega_{j}, X_{s-1} \in \Omega_{i} \}.
\ee 
Define the number of such transitions before time $T$ by 
\be 
N_{i,j}(T) = \max \{ t \geq 0 \, : \, T_{i,j}^{(t)} < T \}.
\ee 

\begin{thm}\label{ThmMainBootstrapBound}
Fix $\frac{1}{2} < 1 - \alpha < \beta < 1$, fix $I \subset [n]$ so that $\pi(\cup_{i \in I} \Omega_{i}) \geq \beta > \frac{1}{2}$, and set $\gamma = \min( \frac{1}{2}, \frac{\alpha + \beta - 1}{\beta}) > 0$. Let $\tau_{\mathrm{wc}}$ be the well-covering time associated with the projected kernel $\overline{K}$. Finally, for $1 \leq i \leq n$, set $\varphi_{i}' = \varphi_{i} \mathbf{1}_{i \in I}$. Then for $T$ satisfying
\be 
T > \tau_{\mathrm{wc}} ( 8 c_{\gamma}' \varphi_{1}', \ldots, 8  c_{\gamma}' \varphi_{n}',  \sqrt{8 \varphi_{\max} \log(64n^{2}T)}),
\ee 
we have 
\be 
\tmix \leq \frac{4}{3} c_{\alpha} T.
\ee 
\end{thm}

We need the following two lemmas. 
\begin{lemma}\label{LemSpreadBound}
Fix notation as in Theorem \ref{ThmMainBootstrapBound}. Then for all $X_{0} = x \in \Omega$, $t \in \mathbb{N}$ and all $0<c<\infty$,
\be \label{IneqSpreadBound1}
\P_{x} \big[ | \frac{1}{t+1}  N_{i,j}(\kappa_{i}^{-1}(t)) -  \overline{K}(i,j)| > c\big]  \leq 4 e^{-\frac{c^{2} (t+1)}{8 \varphi_{\max} }}
\ee 
and
\be \label{IneqSpreadBound2}
\P_{x} \big[ | \frac{1}{t+1}  N_{i,j}(\kappa_{j}^{-1}(t)) -  \overline{K}(j,i)| > c\big] \leq 4 e^{-\frac{c^{2} (t+1)}{8 \varphi_{\max} }}.
\ee 
\end{lemma}
\begin{proof}[Proof of Lemma \ref{LemSpreadBound}]
We first prove inequality \eqref{IneqSpreadBound1}. Define the function $f_{i,j}:\Omega_{i} \rightarrow [0,1]$ by 
\be 
f_{i,j}(y) = \P[X_{1} \in \Omega_{j} \, | \, X_{0} = y].
\ee 
We note $\pi_{i}(f_{i,j}) = \sum_{y \in \Omega_{i}} \pi_{i}(y) f_{i,j}(y) = \overline{K}(i,j)$. By Corollary 2.11 of \cite{Paul14}, for all $x \in \Omega$, $t \in \mathbb{N}$ and $c > 0$ we have
\be \label{IneqApplOfPaulConc}
\P_{x}[| \frac{1}{t+1} \sum_{s=0}^{t} f_{i,j}(X_{s}^{(i)}) -  \overline{K}(i,j)| > c] \leq 2 e^{-\frac{c^{2} (t+1)}{8 \varphi_{\max} }}.
\ee 
Next, denote by $K_{i}$ the transition kernel associated with the Markov chain $\{ X_{t}^{(i)} \}_{t \in \mathbb{N}}$. We may write
\be 
K_{i}(x,\cdot) &= \sum_{j } f_{i,j}(x) K_{i,j}(x,\cdot),
\ee 
where the kernels $K_{i,j}(x,\cdot)$ are defined by
\be 
K_{i,j}(x,\cdot) = \P[X_{1}^{(i)} \in \cdot | X_{0} = x, X_{1} \in \Omega_{j}].
\ee 
For fixed $i,j$, let $D_{i,j}(t) = \textbf{1}_{X_{t} \in \Omega_{i}, X_{t+1} \in \Omega_{j}}$. Thus $\E[D_{i,j}(t) | X_{t} = x \in \Omega_{i}] = f_{i,j}(x)$, and $\sum_{0 \leq s \leq \kappa_{i}^{-1}(t) \, : \, X_{s} \in \Omega_{i}} (D_{i,j}(s) -  f_{i,j}(X_{s}))$ is a martingale relative to the filtration $\sigma(\{X_{s}^{(i)} \}_{s \leq t})$. By Azuma's martingale inequality,
\be 
\P_{x}\big[\frac{1}{t+1}| N_{i,j}(\kappa_{i}^{-1}(t)) - \sum_{s=0}^{t} f_{i,j}(X_{s}^{(i)}) | > c\big] &= \P_{x}\big[\frac{1}{t+1}| \sum_{0 \leq s \leq \kappa_{i}^{-1}(t) \, : \, X_{s} \in \Omega_{i}} (D_{i,j}(s) -  f_{i,j}(X_{s})) | > c\big]\\
&\leq 2e^{-\frac{c^{2}(t+1)}{4}}.
\ee 
Combining this with inequality \eqref{IneqApplOfPaulConc} and using the fact that $\varphi_{\max} \geq 1$, we have shown inequality  \eqref{IneqSpreadBound1}: 
\be 
\P_{x}\big[ | \frac{1}{t+1}  N_{i,j}(\kappa_{i}^{-1}(t)) - \overline{K}(i,j)| > c\big] \leq 2 e^{-\frac{c^{2} (t+1)}{8 \varphi_{\max} }} + 2e^{-\frac{c^{2}(t+1)}{4}} \leq 4 e^{-\frac{c^{2} (t+1)}{8 \varphi_{\max} }}.
\ee 
Inequality \eqref{IneqSpreadBound2} follows immediately by applying inequality \eqref{IneqSpreadBound1} to the time-reversal of the chain $\{X_{t} \}_{t \in \mathbb{N}}$ and the proof is finished. 
\end{proof}

\begin{lemma}[Local-to-Global Spreading] \label{LemmaLocToGlob}
Fix notation as in Theorem \ref{ThmMainBootstrapBound}. Then for any $B>0$ and $T > \tau_{\mathrm{wc}}(B \varphi_{1},\ldots, B \varphi_{n}, \sqrt{8 \varphi_{\max} \log( \frac{8n^{2} T}{ \epsilon})})$, 
\be 
\max_{x \in \Omega} \P_{x}[\max_{i \in [n]} \frac{\kappa_{i}(T)}{\varphi_{i}} < B] \leq \epsilon.
\ee 
\end{lemma}

\begin{proof}
Fix $T \in \mathbb{N}$, $X_{0} = x \in \Omega$ and $\epsilon > 0$. For $1 \leq i \neq j \leq n$, denote by $\mathcal{A}_{i,j}$ the event that 
\be 
|N_{i,j}(T) - \kappa_{i}(T) \overline{K}(i,j) | > \sqrt{8 \varphi_{\max} \log( \frac{8n^{2} T}{ \epsilon})\kappa_{i}(T)}
\ee 
and denote by $\mathcal{B}_{i,j}$ the event that
\be 
|N_{i,j}(T) - \kappa_{j}(T) \overline{K}(j,i) | > \sqrt{8 \varphi_{\max} \log( \frac{8n^{2} T}{ \epsilon})\kappa_{i}(T)}.
\ee 
By Lemma \ref{LemSpreadBound},
\be 
\P_{x}[\mathcal{A}_{i,j}] &= \P_{x}[|N_{i,j}(T) -  \kappa_{i}(T) \overline{K}(i,j) | > \sqrt{8 \varphi_{\max} \log( \frac{8 n^{2} T}{\epsilon}) \kappa_{i}(T)} ] \\
&= \sum_{t=0}^{T} \P_{x}[|N_{i,j}(T) -  \kappa_{i}(T) \overline{K}(i,j) | > \sqrt{8 \varphi_{\max} \log( \frac{8 n^{2} T}{\epsilon})\kappa_{i}(T)} | \kappa_{i}(T) = t] \P_{x}[\kappa_{i}(T) = t] \\
&= \sum_{t=0}^{T} \P_{x}[ | \frac{1}{t}  N_{i,j}(\kappa_{i}^{-1}(t)) -  \overline{K}(i,j)| > \frac{ \sqrt{8 \varphi_{\max} \log(\frac{8 n^{2} T}{\epsilon})}}{\sqrt{t}} | \kappa_{i}(T) = t]  \P_{x}[\kappa_{i}(T) = t] \\
&\leq \P_{x}[ \max_{0 \leq t \leq T} \sqrt{t} | \frac{1}{t}  N_{i,j}(\kappa_{i}^{-1}(t)) -  \overline{K}(i,j)|  > \sqrt{8 \varphi_{\max} \log(\frac{8 n^{2} T}{\epsilon})}] \\
&\leq  \frac{\epsilon}{2n^{2}}.
\ee 
Taking a union bound yields
 $\P_{x}[\cup_{1 \leq i \neq j \leq n} \mathcal{A}_{i,j}] \leq {\epsilon \over 2}$ and thus $\P_{x}[\cap_{1 \leq i \neq j \leq n} \mathcal{A}^c_{i,j}] > 1- {\epsilon \over 2}$. The same calculation implies that $\P_{x}[\cap_{1 \leq i \neq j \leq n} \mathcal{B}^c_{i,j}] >1 - \frac{\epsilon}{2}$. By construction, we also have
\be 
| \sum_{i} N_{i,j}(T) - \kappa_{j}(T) |, \, | \sum_{j} N_{i,j}(T) - \kappa_{i}(T) | \leq 1. 
\ee 
Set
\be
\tilde{N}(i,j) = {N_{i,j}(T) \over T}, \, \tilde{\kappa}(i) = {\kappa_i(T) \over T}.
\ee
We have shown that, with at least $1-{\epsilon\over 2}$ probability, the pair $(\tilde{\kappa},\tilde{N})$ belongs to the set $\mathcal{S}_{B',T}$ associated with the kernel $\bar{K}$ and constant $B' = \sqrt{8 \varphi_{\max} \log( \frac{8 n^{2} T}{\epsilon})}$. 
 Thus, by the definition of the well-covering time, for $T > \tau_{\mathrm{wc}}(B \varphi_{1},\ldots, B \varphi_{n}, \sqrt{8 \varphi_{\max} \log( \frac{8 n^{2} T}{\epsilon})})$, with at least  $1- {\epsilon \over 2}$ probability, 
 $\tilde{\kappa}(i) > {\varphi_i \over B}$. This immediately yields that 
\be 
\P_{x}[\max_{i \in [n]} \frac{\kappa_{i}(T)}{\varphi_{i}} < B] \leq \frac{\epsilon}{2}.
\ee 
Since $x \in \Omega$ was arbitrary, the proof is finished.
\end{proof}

We finally give
\begin{proof} [Proof of Theorem \ref{ThmMainBootstrapBound}]
By Lemma \ref{LemmaLocToGlob}, for $T > \tau_{\mathrm{wc}} ( 8 c_{\gamma}' \varphi_{1}', \ldots, 8 c_{\gamma}' \varphi_{n}',  \sqrt{8 \varphi_{\max} \log(64n^{2} T)})$, 
\be 
\max_{x \in \Omega} \P_{x}[\max_{i \in I} \frac{\kappa_{i}(T)}{\varphi_{i}} < 8  c_{\gamma}'] \leq \frac{1}{8}.
\ee 
Thus, by Lemma \ref{LemmaBasicMixing},
\be 
\tmix \leq \frac{4}{3} c_{\alpha} T
\ee 
and the proof is finished. 
\end{proof}

Combining inequality \eqref{IneqFirstWellCoveringBound} with Theorem \ref{ThmMainBootstrapBound}, we see that if a Markov chain $\{X_{t} \}_{t \in \mathbb{N}}$ has $Q$ as its projected chain, and the mixing times of all restricted chains are less than $\varphi_{\max}$, the mixing time of $\{ X_{t} \}_{t \in \mathbb{N}}$ satisfies
\be \label{EqConclusionWellCoveringSimpleExample}
\frac{\tmix}{\log(\tmix)} = O( n \log(n) \Delta^{2} D^{2} \varphi_{\max}).
\ee 

\subsection{Comparison inequalities for well-covering times} \label{SubsecCompWellCov}

Like the spectral gap, the covering time can be difficult to bound for generic Markov chains. One of the main tools for obtaining quantitative bounds on the spectral gap of a Markov chain is the use of comparison theorems to relate complicated chains of interest to simpler chains that can be analyzed directly (see \textit{e.g.,} \cite{DiSa93b, DGJM06}). In this section, we give some basic comparison results for the well-covering time, for the same purpose. The following `scaling' bounds are immediate: 
\begin{itemize}
\item For any $\alpha > 1$, and any $t_{1},\ldots,t_{n},B$,
\be \label{IneqCovScale1}
\tau_{\mathrm{wc}}(\alpha t_{1},\ldots, \alpha t_{n}, B) \leq \alpha \tau_{\mathrm{wc}}(t_{1},\ldots,t_{n}, B).
\ee
\item For any $\alpha > 1$, $B >0$ and $T \in \mathbb{N}$, 
\be 
\mathcal{S}_{\alpha B, \alpha^{2} T} \subset \mathcal{S}_{B,T},
\ee 
and so for any $t_{1},\ldots,t_{n}$, 
\be \label{IneqCovScale2}
\tau_{\mathrm{wc}}(t_{1},\ldots,  t_{n}, \alpha B) \leq \alpha^{2} \tau_{\mathrm{wc}}(t_{1},\ldots,t_{n}, B).
\ee

\end{itemize}
The next result shows that the well-covering time of a complicated kernel can be bounded in terms of the well-covering time of simpler kernels:

\begin{lemma} \label{LemmaMonotone}
Let $Q,Q'$ be two reversible kernels on $[n]$ with the same stationary measure $\mu$, and let $\tau_{\mathrm{wc}}$, $\tau_{\mathrm{wc}}'$ be their well-covering times. Assume that $Q(i,j) \geq Q'(i,j)$ for all $j \neq i$. Then for any sequence $t_{1},\ldots,t_{n},B$, 
\be \label{IneqMonotonicity}
\tau_{\mathrm{wc}}(t_{1},\ldots,t_{n},B) \leq 9 \tau_{\mathrm{wc}}'(t_{1},\ldots,t_{n},B).
\ee 
\end{lemma}
\begin{proof}
We begin by showing that, under the same assumptions,
\be \label{IneqMonotonicity}
\tau_{\mathrm{wc}}(t_{1},\ldots,t_{n},B) \leq \tau_{\mathrm{wc}}'(t_{1},\ldots,t_{n},3B).
\ee 
Let $\mathcal{S}_{B,T}$ and $\mathcal{S}_{B,T}'$ denote the pairs $(\kappa,N)$ that satisfy inequalities \eqref{ineqWellCovDef} for the kernels $Q,Q'$ respectively. To prove inequality \eqref{IneqMonotonicity}, it is enough to find, for any pair $(\kappa, N) \in \mathcal{S}_{B,T}$ that satisfies $\min_{i} \frac{\kappa(i)}{t_{i}} < \frac{1}{T}$, some pair $(\kappa', N') \in \mathcal{S}_{3B,T}'$ that satisfies $\min_{i} \frac{\kappa'(i)}{t_{i}} < \frac{1}{T}$. \\
We now give such a construction. Set 
\be \kappa'(i) = \kappa(i) 
\ee 
for all $1 \leq i \leq n$, set 
\be 
N'(i,j) = N(i,j) - \kappa(i)(Q(i,j) - Q'(i,j))
\ee 
for all $1 \leq i \neq j \leq n$, and finally set
\be 
N'(i,i) = \kappa'(i) - \sum_{j \neq i} N'(i,j).
\ee 
Since $\min_{i} \frac{\kappa(i)}{t_{i}} < \frac{1}{T}$, it is clear that $\min_{i} \frac{\kappa'(i)}{t_{i}} < \frac{1}{T}$. Thus, it just remains to check that $(\kappa', N') \in \mathcal{S}_{B,T}'$ by confirming that they satisfy both parts of inequality \eqref{ineqWellCovDef}. To check the first part of line 1 of inequality \eqref{ineqWellCovDef}, 
\be 
|N'(i,j) - \kappa'(i) Q'(i,j) | &= |N(i,j) - \kappa(i)(Q(i,j) - Q'(i,j)) - \kappa(i) Q'(i,j)| \\
&= |N(i,j) - \kappa(i) Q(i,j)| \\
&< \frac{B \sqrt{\kappa(i)}}{\sqrt{T}} =\frac{B \sqrt{\kappa'(i)}}{\sqrt{T}}. \\
\ee 
To check the second part of line 1 of inequality \eqref{ineqWellCovDef}, note that by reversibility and then the first part of inequality \eqref{ineqWellCovDef},
\be 
Q(j,i) | \kappa(i) \frac{\mu(j)}{\mu(i)} - \kappa(j) | = |\kappa(i) Q(i,j) - \kappa(j) Q(j,i) | \leq  \frac{2 B \sqrt{\kappa(i)}}{\sqrt{T}}.
\ee 
Thus,
\be 
Q'(j,i) | \kappa'(i) \frac{\mu(j)}{\mu(i)} - \kappa'(j) | \leq \frac{2 B \sqrt{\kappa'(i)}}{\sqrt{T}}.
\ee 
We conclude that
\be 
|N'(i,j) - \kappa'(j) Q'(j,i) | &= |N(i,j) - \kappa(i)(Q(i,j) - Q'(i,j)) - \kappa(j) Q'(j,i) | \\
&= |N(i,j) - \kappa(i) Q(i,j) + \kappa(i) Q'(i,j) - \kappa(j) Q'(j,i) | \\
&\leq  |N(i,j) - \kappa(i) Q(i,j)| + Q'(j,i) | \kappa'(i) \frac{\mu(j)}{\mu(i)} - \kappa'(j) |  \leq \frac{3 B \sqrt{\kappa'(i)}}{\sqrt{T}}.
\ee 
The second part of inequality \eqref{ineqWellCovDef} is immediate. This completes the proof of inequality \eqref{IneqMonotonicity}; the result now follows from inequality \eqref{IneqCovScale2}.
\end{proof}

In the other direction, making a chain lazier cannot greatly impact the well-covering time. This requires an intermediate lemma; we give an abbreviated proof, as the details may be checked exactly as in the proof of Lemma \ref{IneqMonotonicity}:  

\begin{lemma} [Well-Behaved Covering Set] \label{LemmaBehaving}
Define 
\be \label{EqDefNiceCovColl}
\mathcal{R}_{B, T} \equiv \{ (\kappa, N) \in \mathcal{S}_{B, T} \, : \, \max_{i \neq j} \frac{N(i,j)}{Q(i,j)} \leq 1 \}. 
\ee
Then
\be 
\min  \Big \{ T > 0 \, : \, \forall \,& (\kappa, N) \in \mathcal{S}_{B,T}, \, \forall \, i \in [n], \, \kappa(i) > \frac{t_{i}}{T} \Big \} \\
&= \min \Big \{ T > 0 \, : \, \forall \, (\kappa, N) \in \mathcal{R}_{B,T}, \, \forall \, i \in [n], \, \kappa(i) > \frac{t_{i}}{T} \Big \}.
\ee 
\end{lemma} 
\begin{proof}
Since $\mathcal{R}_{B,T} \subset \mathcal{S}_{B,T}$, it is clear that the right-hand side is at least as large as the left-hand side. To prove the reverse inequality, we define a map $F = (F_{1}, F_{2})$ from $(\kappa, N) \in \mathcal{S}_{B,T}$ to $\mathcal{R}_{B,T}$ by setting $F_{1}(\kappa) = \kappa$, $F_{2}(N)(i,j) = \min (N(i,j), Q(i,j))$ for $i \neq j$ and then $F_{2}(N)(i,i) = 1 - \sum_{j \neq i} F_{2}(N)(i,j)$. This map sends elements of  $ \mathcal{S}_{B,T}$ to $\mathcal{R}_{B,T}$. Also, if $\min_{i} \frac{\kappa(i)}{t_{i}} < \frac{1}{T}$, then $\min_{i} \frac{F_{1}(\kappa)(i)}{t_{i}} < \frac{1}{T}$. This completes the proof.
\end{proof}

The following  result goes in the `opposite direction' from Lemma \ref{LemmaMonotone}: 
\begin{lemma} [Laziness and Well-Covering Times] \label{LemmaLazy}
Let $Q$ be a $\frac{1}{2}$-lazy reversible kernel with stationary measure $\mu$, let $\mathrm{Id}$ be the identity kernel, let $0 < \alpha < 1$, and let $Q' = \alpha Q + (1-\alpha) \mathrm{Id}$.  Let $\tau_{\mathrm{wc}}$, $\tau_{\mathrm{wc}}'$ be the well-covering times of $Q,Q'$.  Then for any sequence $t_{1},\ldots,t_{n},B$, 
\be \label{IneqLaziness}
\tau_{\mathrm{wc}}'(t_{1},\ldots,t_{n},B) \leq \alpha^{-2} \tau_{\mathrm{wc}}(t_{1},\ldots,t_{n},B).
\ee 
\end{lemma}
\begin{proof}
Let $\mathcal{S}_{B,T}$ and $\mathcal{S}_{B, T}'$ denote the pairs $(\kappa,N)$ that satisfy inequalities \eqref{ineqWellCovDef} for the kernels $Q,Q'$ respectively and let $\mathcal{R}_{B,T}$ and $\mathcal{R}_{B, T}'$ be as in Equation \eqref{EqDefNiceCovColl}. We then define a bijection $F = (F_{1}, F_{2})$ from $\mathcal{R}_{B,T}$ to $\mathcal{R}_{\alpha B, T}'$ by setting $F(\kappa, N) = (\kappa',N')$ where
\be 
\kappa'(i) = \kappa(i)
\ee 
for all $1 \leq i \leq n$,
\be 
N'(i,j) = \alpha N(i,j)
\ee 
for all $i \neq j$, and $N'(i,i) = 1 - \sum_{j \neq i} N'(i,j)$ for all $1 \leq i \leq n$. This map is injective, and its image is contained in $\mathcal{R}_{\alpha B, T}$. To check that it is in fact bijective, we define a map $F^{-1} = (F_{1}^{-1}, F_{2}^{-1})$ from $\mathcal{R}_{\alpha B, T}$ to $\mathcal{R}_{B,T}$ by setting $F^{-1}(\kappa',N') = (\kappa,N)$ where $\kappa(i) = \kappa(i)'$
for all $1 \leq i \leq n$, $
N(i,j) = \alpha^{-1} N(i,j)$ 
for all $i \neq j$, and  $N(i,i) = 1 - \sum_{j \neq i} N(i,j)$ for all $1 \leq i \leq n$. It can be verified that $F^{-1}$ is an injection and that $F \circ F^{-1}$ is the identity. Since $F_{1}(\kappa) = \kappa$, this implies that
\be 
\min \Big \{ T > 0 \, : \, &\forall \, (\kappa, N) \in \mathcal{R}_{B,T}, \, \forall \, i \in [n], \, \kappa(i) > \frac{t_{i}}{T} \Big \}\\
& = \min \Big \{ T > 0 \, : \, \forall \, (\kappa, N) \in \mathcal{R}_{\alpha B,T}', \, \forall \, i \in [n], \, \kappa(i) > \frac{t_{i}}{T} \Big \}.
\ee 
By Lemma \ref{LemmaBehaving}, this implies 
\be 
\tau_{\mathrm{wc}}'(t_{1},\ldots,t_{n},\alpha B) \leq  \tau_{\mathrm{wc}}(t_{1},\ldots,t_{n},B).
\ee 
Combining this with inequality \eqref{IneqCovScale2} completes the proof.
\end{proof}

As mentioned before, Lemmas \ref{LemmaMonotone} and \ref{LemmaLazy} are meant to be simple analogues of the well-developed comparison theory for Markov chains \cite{DGJM06}. The bounds in this note can already be combined with Lemma \ref{LemmaWellCovEx} to obtain at least some bound on the well-covering time of any irreducible $\frac{1}{2}$-lazy Markov chain, though this bound is often very conservative. For example, if the Markov chain exhibits drift towards a small number of states (\textit{e.g.}, the KCIP chain in Example \ref{ExKcipCarefulStatement}), the associated well-covering time can be much closer to the mixing time of the kernel $Q$ than would be suggested by comparison with Lemma \ref{LemmaWellCovEx}. This same strong dependency of our bounds on the stationary distribution of the underlying Markov chain occurs for the usual comparison theory as well. 

\section{Stronger Mixing Bounds with Additional Regularity} \label{SecSpecCond}

We discuss additional assumptions that can give stronger bounds on the mixing time, with an emphasis on bounds that are effective before the occupation measures of `most' parts of the partition are large. These bounds are most useful when $n$ is large.

\subsection{Drift Bound} \label{SecDriftBound}
One of the main difficulties in using the bounds in \cite{MaRa00,JSTV04,MaRa02,MaRa06, MaYu09}, as well as our bounds in Section \ref{SecNaiveBoot}, is their sensitivity to poor mixing on sets that have small measure under $\pi$. The simplest way to circumvent this difficulty is through a `drift condition.'

Neither drift conditions nor attempting to ignore sets of small measure when bounding mixing times are new ideas; we discuss them here because they are popular and useful in the context of this paper, not novel. Drift conditions were famously used in \cite{Rose95} and many subsequent papers to derive general mixing bounds for chains. A central part of the probabilistic bound on the mixing time given in \cite{BSZ10} involves showing that certain sets of small measure can (eventually) be ignored, and \cite{kovchegov2015path} explicitly discusses this issue in the context of path-coupling arguments (see \cite{BuDe97}). The literature on ignoring sets of small measure when proving Poincar\'{e} and log-Sobolev inequalities seems smaller (however, see  \cite{Schw02} for one example).  \par

Fix constants $0 < a \leq 1$, $0 \leq b < \infty,$ and $k \in \mathbb{N}$ and let $V: \Omega \rightarrow \mathbb{R}^{+}$ be a function that satisfies the drift condition
\be \label{IneqDriftCondBasic}
\E[V(X_{t+k}) | X_{t}] \leq (1 - a) V(X_{t}) + b
\ee 
and has $\max_{x \in \Omega} V(x) = \vmax < \infty$. Inequality \eqref{IneqDriftCondBasic} is a special case of the popular drift condition used in \cite{Rose95}, and the function $V$ is often called a \textit{Lyapunov function}. Define the sets 
\be \label{EqDefSmallSets}
\L(C) = \{ \omega \in \Omega \, : \, V(\omega) \leq C \}.
\ee 
We have:

\begin{thm} [Decompositions and Drift Condition] \label{ThmDecompDrift}
Let $K$ be a transition kernel with state space $\Omega$ and let $V, a,b,k$ satisfy inequality \eqref{IneqDriftCondBasic}. Fix $\frac{4a}{b} < M < \infty$ and let $\Omega' = \L(M)$. Finally, let $\tmix'$ be the mixing time of the trace of $K$ on $\Omega'$. Then the mixing time $\tmix$ of $K$ satisfies:
\be 
\tau_{\mathrm{mix}} \leq \frac{16 c_{\gamma}}{3a} \ \max(\frac{16}{ c_{\gamma}'} \tmix', k \log(16 V_{\max}), 8 \log(16)).
\ee 
\end{thm}

\begin{proof}
The proof is given in Appendix B.
\end{proof}

\subsection{Regularity and Contractivity Assumptions} \label{SecRegAssump}

One of the main contributions of \cite{JSTV04} was the use of regularity assumptions to strengthen their bounds. In this section, we consider one useful and strong assumption that has been satisfied in practice (see \textit{e.g.}, Lemma 4.5 of \cite{DLP10}) \footnote{See Corollary \ref{CorSimpSupBdBad2} and Lemma \ref{LemmaSimpSupBdBad2} for a simple bound based on a regularity condition that looks more similar to the bounds in \cite{JSTV04}.}. Our assumptions in this section are closely related to the notion of \textit{metastability}; see \textit{e.g.} the very recent \cite{zhang2015multi} for bounds on the spectral gap and log-Sobolev constants of metastable chains that are useful in similar situations.

Define a less lazy version of $\overline{K}$ by setting 
\be \label{EqDefProjChainLL}
\overline{K}_{\mathrm{LL}}(i,j) = \frac{\overline{K}(i,j)}{2(1 - \overline{K}(i,i))}
\ee 
for $i \neq j$ and $\overline{K}_{\mathrm{LL}}(i,i) = 1 - \sum_{j \neq i} \overline{K}_{\mathrm{LL}}(i,j)$.

For any pair of measures $\mu,\nu$ on a metric space $(\mathcal{X},d)$, denote by $\Pi(\mu,\nu)$ the collection of all pairs of random variables $(X,Y) \in \mathcal{X}^{2}$ that have marginal distributions $X \stackrel{D}{=} \mu$, $Y \stackrel{D}{=} \nu$. Recall that the \textit{Wasserstein metric} on measures on a metric space $(\mathcal{X}, d)$ is given by
\be 
W_{d}(\mu,\nu) = \inf_{(X,Y) \in \Pi(\mu,\nu)} \E[d(X,Y)].
\ee 
Recall the escape time $\tEsc[i]$ from \eqref{eqn:escapetime}.

\begin{defn} [Contraction Condition] 
Let $d$ be a metric on $[n]$. For $X_{0} = x \in \Omega_i$, let $\mu_{x}'$ be the distribution of $\mathcal{P}(X_{\tEsc[i]}) \in [n] \backslash \{i\}$ and let $\mu_{x}(\cdot) = \frac{1}{2} \mu_{x}'(\cdot)+ \frac{1}{2} \delta_{i}(\cdot)$. Say that the kernel $K$ satisfies a contraction condition with coefficients $0 < \beta < \alpha \leq 1$ if 
\be \label{IneqContractionAssumption}
\max_{x \in \Omega_{i}, y \in \Omega_{j}}  W_{d}(\mu_{x}, \mu_{y}) \leq (1-\alpha) d(i,j) + \beta.
\ee 
\end{defn}

The parameter $\frac{\beta}{\alpha} < 1$ plays a role in Theorem \ref{ThmContCondHighDim} similar to the role of the regularity parameter $\gamma$ in Theorem 1 of \cite{JSTV04}. The purpose of the Definitions \eqref{EqDefProjChainLL} and \eqref{IneqContractionAssumption} is to allow us to couple a suitable sped-up copy of the function $\{ \mathcal{P}(X_{t}) \}_{t \geq 0}$ to a Markov chain $\{Z_{t}\}_{t \geq 0}$ evolving according to $\overline{K}_{\mathrm{LL}}$ so that $d(\mathcal{P}(X_{t}), Z_{t})$ is often small; this is made precise in inequality \eqref{IneqUsingContractCondFirst}.

We obtain the following bound for Markov chains satisfying  \eqref{IneqContractionAssumption}:

\begin{thm} \label{ThmContCondHighDim}

Let $K$ be a Markov chain on state space $\Omega = \sqcup_{i=1}^{n} \Omega_{i}$ and let $d$ be a metric on $[n]$  that satisfies $1 \leq d(i,j) \leq D_{\max} < \infty$ for all $i \neq j$. Assume that $K$ satisfies inequality \eqref{IneqContractionAssumption} for some $0 < \beta < \frac{\alpha}{2} \leq \frac{1}{2}$ and that it also satisfies
\be [IneqRegularityAssumptions2]
 \min_{x \in \Omega_{i}} \P_{x}[\tEsc[i] > a_{1} \varphi_{\max} \log(n)]&\geq \delta_{1} \\
 \max_{x \in \Omega_{i}} \P_{x}[\tEsc[i] > a_{2}  \varphi_{\max} \log(n)] &\leq 1- \delta_{2}
\ee 
for some  $0 < a_{1}, a_{2}, \delta_{1},\delta_{2}$. Then the mixing time $\tmix$ of $K$ satisfies 
\be \label{IneqContCondConcMixing}
\tmix \leq \tilde{C}_{1} \, \varphi_{\max} \log(n) \, \max( \tilde{C}_{2} \overline{\varphi}  + 1,  \frac{\tilde{C}_{3}}{\log(1 - \alpha)}),
\ee 
where $\overline{\varphi}$ is the mixing time of the kernel $\overline{K}_{\mathrm{LL}}$ defined in \eqref{EqDefProjChainLL},  $\tilde{C}_{1} = \frac{1024 }{\gamma}   c_{2 \epsilon} \, \frac{a_{2}}{\delta_{2}}   \log(16) \big( \log(1 - \delta_{1}^{ \lceil \frac{8  e   }{a_{1} \,  c_{\epsilon}' } \rceil}) \big)^{-1}$, $\gamma = \frac{1}{2} - \frac{\beta}{\alpha}$, $\tilde{C}_{2} = \log_{2}(\frac{8}{\gamma})$, $\tilde{C}_{3} = \log(\frac{8}{\gamma}) + \log(D_{\max})$ and $\epsilon = \frac{1}{4} - \frac{\gamma}{16}$.
\end{thm}

\begin{remarks}
We point out that this result is easier to apply than it might appear at first glance:
\begin{itemize}
\item Since $a_{1}, a_{2}$ are arbitrary (and can depend on $n$), an inequality of the form \eqref{IneqRegularityAssumptions2} will be satisfied for any ergodic Markov chain on a finite state space.
\item The popular \textit{Total Variation} distance is in fact a Wasserstein distance. Thus, Inequality \eqref{IneqContractionAssumption} will also be satisfied by all sufficiently large powers $K^{k}$ of any ergodic Markov chain $K$ on a finite state space. See Example \ref{RemContCondTrace2} for a general approach to proving such inequalities for more natural metrics and with $k=1$.
\end{itemize}
\end{remarks}

\begin{proof}
We begin by constructing a coupling of the Markov chain $\{X_{t} \}_{t \geq 0}$, with state space $\Omega$, to a Markov chain evolving according to $\overline{K}_{\mathrm{LL}}$ on state space $[n]$.

Fix $X_{0} = x$, let $\tau_{\mathrm{exit}}^{(0)} = 0$, and define inductively $\tau_{\mathrm{exit}}^{(s+1)} = \min\{ t > \tau_{\mathrm{exit}}^{(s)} \, : \, \mathcal{P}(X_{t}) \neq \mathcal{P}(X_{\tau_{\mathrm{exit}}^{(s)}}) \}$. For $t \in \mathbb{N}$, define $Y_{t}' = \mathcal{P}(X_{\tau_{\mathrm{exit}}^{(t)}})$. Let $\{\eta_{t} \}_{t \geq 0}$ be a sequence of i.i.d. geometric random variables with mean 2, let $\lambda^{(s)} = \min \{ j \geq 0 \,: \, \sum_{i=1}^{j} \eta_{i} \geq s \}$, and for $\lambda^{(s)} \leq t < \lambda^{(s+1)}$, define $Y_{t} = Y_{s}'$. We note that $\{ Y_{t} \}_{t \geq 0}$ is not a Markov chain.

Denote by $\{Z_{t} \}_{t \in \mathbb{N}}$ a Markov chain on $[n]$ evolving according to the kernel $\overline{K}_{\mathrm{LL}}$ and started according to the distribution $\overline{\pi}[i] \equiv \pi(\Omega_{i})$. By the assumption made in Equation \eqref{IneqContractionAssumption}, it is possible to couple $\{Y_{t}\}_{t \in \mathbb{N}}$, $\{Z_{t}\}_{t \in \mathbb{N}}$ so that
\be \label{IneqUsingContractCondFirst}
\E[d(Y_{t+1}, Z_{t+1}) | Y_{t}, Z_{t}] \leq (1- \alpha) d(Y_{t}, Z_{t}) + \beta.
\ee 
Under this coupling, for any $t \in \mathbb{N}$,
\be 
\P[Y_{t} = Z_{t}] &\geq \P[d(Y_{t},Z_{t}) < 1] \\
&\geq 1 - \E[d(Y_{t},Z_{t})] \\
&\geq 1 - \frac{\beta}{\alpha} - (1-\alpha)^{t} D_{\max}.\label{IneqWhyNoLabel}
\ee

Fix any subset $I \subset [n]$ satisfying $\pi(\cup_{i \in I}\Omega_{i}) > \frac{1}{2} - \frac{\gamma}{8}$ and let $\overline{\tau}_{I} = \min \{t > 0 \, : \, Y_{t} \in I \}$. Then for any starting points $Y_{0} = y, Z_{0} =z$ and any $T \geq \max(\frac{\overline{\varphi}}{\log(2)} \log(\frac{\gamma}{8}), \frac{\log(\gamma) - \log(8D)}{\log(1-\alpha)})$, we have by inequality \eqref{IneqWhyNoLabel}  
\be 
\P[\overline{\tau}_{I} \leq T] &\geq \P[Y_{T} \in I] \\
&\geq \P[Z_{T} \in I] - \P[Y_{T} \neq Z_{T}]  \\
&\geq (\pi(\cup_{i \in I} \Omega_{i}) - 2^{-\lfloor \frac{T}{\overline{\varphi}} \rfloor}) - (\frac{\beta}{\alpha} + D_{\max} (1-\alpha)^{T}) \\
&\geq (\frac{1}{2} - \frac{\gamma}{8} - \frac{\gamma}{8}) - (\frac{1}{2} - \gamma - \frac{\gamma}{8}) \geq \frac{\gamma}{4}. 
\ee 
Since this holds uniformly over initial points $Y_{0},Z_{0}$, we have for $k \in \mathbb{N}$
\be \label{IneqStochDomHittingProj}
\max_{y \in [n] } \P_{y}[\overline{\tau}_{I} \geq k \frac{4T}{\gamma}] \leq e^{-k}.
\ee 

Let $\tau_{I} = \min \{t > 0 \, : \, X_{t} \in \cup_{i \in I} \Omega_{i} \}$. Combining inequalities \eqref{IneqRegularityAssumptions2} and \eqref{IneqStochDomHittingProj}, we have for $k \in \mathbb{N}$   
\be 
\max_{x \in \Omega} \P_{x}[\tau_{I} > k \frac{16T}{\gamma}  \frac{a_{2} \varphi_{\max}}{\delta_{2}} \log(n)] &\leq \max_{x \in \Omega} \P_{x}[\overline{\tau}_{I} \geq k \frac{4T}{\gamma}] + \max_{x \in \Omega} \P_{x}[\tau_{\mathrm{exit}}^{( \lceil k \frac{4T}{\gamma} \rceil)} > k \frac{16T}{\gamma}  \frac{a_{2} \varphi_{\max}}{\delta_{2}} \log(n) ] \\
&\leq e^{-k} + e^{- \lfloor \frac{4kT}{\gamma} \rfloor} \\
&\leq 2 e^{-k}, \label{IneqStochDomHittingLift}
\ee  
where the second-last line follows from standard concentration inequalities for i.i.d. sums of geometric random variables.
For $C \in \mathbb{N}$, let
\be 
\tau_{I,\mathrm{cov}}(C) = \min \{t > 0 \, : \, \sum_{s: \tau_{\mathrm{exit}}^{(s)} < t} \textbf{1}_{Y_{s} \in I} > C
 \}
\ee 
be the first time that $X_{t}$ has entered $\cup_{i \in I} \Omega_{i}$ at least $C$ times. By inequalities \eqref{IneqStochDomHittingLift} and \eqref{eqn:subexp}, we have 
\be \label{IneqBigNCov1}
\max_{x \in \Omega} \E_{x}[\tau_{I,\mathrm{cov}}(C)] \leq \frac{48 T}{\gamma}  \frac{a_{2} \varphi_{\max}}{\delta_{2}} \log(n) C.
\ee 

By inequalities \eqref{IneqRegularityAssumptions2} and \eqref{eqn:subexp2}, 
\be 
\P_{x}[ \forall i \in I, \, \kappa_{i}(\tau_{I,\mathrm{cov}} (C)) \leq  \frac{e \,   \varphi_{\max}}{c_{\epsilon}'} \log(8n)] &\leq \max_{x \in \Omega} \P_{x}[ \tau_{\mathrm{exit}}^{(1)} \leq  \frac{e \,   \varphi_{\max} }{c_{\epsilon}'} \log(8n)]^{C} \\
&= (1 - \min_{x \in \Omega} \P_{x}[ \tEsc[1] > \frac{ e \,  \varphi_{\max}}{ c_{\epsilon}'} \log(8n)])^{C} \\
&\leq (1 - \delta_{1}^{ \lceil \frac{8  e   }{a_{1} \,  c_{\epsilon}'} \rceil} )^{C} \label{IneqBigNCov2}
\ee 
for all $C \in \mathbb{N}$.

Combining inequalities \eqref{IneqBigNCov1} and \eqref{IneqBigNCov2} with Markov's inequality and setting $C = \log(16) \big( \log(1 - \delta_{1}^{ \lceil \frac{8  e   }{a_{1} \,  c_{\epsilon}' } \rceil}) \big)^{-1}$, for all 
$t > \frac{768 T}{\gamma}  \frac{a_{2} \varphi_{\max}}{\delta_{2}} \log(n) C$, we have that 

\be 
\max_{x \in \Omega} \P_{x}[\forall i \in I, \, \kappa_{i}(t) < \frac{e  \varphi_{\max}}{  c_{\epsilon}'} \log(8n)] \leq \frac{1}{8}.
\ee 
Since this applies for all $I \subset [n]$ that satisfy  $\pi(\cup_{i \in I}\Omega_{i}) > \frac{1}{2} - \frac{\gamma}{8} = \frac{1}{2} - \frac{1}{8}(\frac{1}{2} - \frac{\beta}{\alpha})$, the result now follows from Lemma \ref{LemmaBasicMixing2}.
\end{proof}

\begin{example} \label{RemContCondTrace2}
The constants $\alpha, \beta$ associated with inequality \eqref{IneqContractionAssumption} are generally very poor for any partition of $\Omega$. However, in some situations, a trace of the Markov chain onto a set with large stationary measure will satisfy inequality \eqref{IneqContractionAssumption} with much larger constants. We give a prototypical example for which this small trick is useful, beginning with a discussion of why the trick is needed. We leave the proof of all of the claims made in this example to Appendix C.

Fix integers $\ell, m \geq 2$, let $\Omega = \mathbb{Z}_{2 \ell}^{m} = \{0,1,2,\ldots,2 \ell-1\}^{m}$ be the $m$-dimensional torus with side length $2 \ell$, and let $Q$ be the proposal distribution 
\be
Q((x_{1},\ldots,x_{m}), (y_{1},\ldots,y_{m})) =  \frac{1}{3 m} \sum_{j=1}^{m} \textbf{1}_{\forall \, i \neq j, \, x_{i} = y_{i}} \textbf{1}_{|x_{j} - y_{j}| \leq 1}, 
\ee
where addition is taken in the group $\mathbb{Z}_{2 \ell}^{m}$. Define the function $H$ on $\Omega$ by
\be 
H(x_{1},\ldots,x_{m}) = \sum_{i=1}^{m} \min(x_{i}, 2\ell -1 - x_{i}).
\ee 
Next, fix $C > 1$, let 
\be \label{eqn:targ}
\pi(x) \propto e^{-C H(x) \log(m)}
\ee 
be a distribution, let $K$ be the kernel of a Metropolis-Hasting Markov chain with proposal kernel $Q$ and target distribution $\pi$, and for $z \subset [m]$ (we allow $z = \emptyset$ as well) define 
\be 
\Omega_{z} = \{ x \in \Omega \, : \, \forall i \notin z, \, x_{i} \leq \ell-1; \, \forall i \in z, \, x_{i} \geq \ell\}.
\ee 
Thus, $\Omega = \sqcup_{z \subset [m]} \Omega_{z}$. Finally, let $d(z,z') = | z \Delta z' |$ be the usual Hamming distance on subsets of $[m]$. 
We are interested in the mixing of the above Markov chain when $k,\ell$ and $C > 6$ are held constant and $m$ goes to infinity. 

We give an informal argument that inequality \eqref{IneqContractionAssumption} cannot be satisfied with useful constants.  Let $\{ X_{t} \}_{t \geq 0}$, $\{ Y_{t} \}_{t \geq 0}$ be two copies of the Markov chain started at points points $x_{i} = (\ell-1)\textbf{1}_{i < \frac{m}{2}} \in \Omega_{\emptyset}$ and $y_{i} = (\ell-1)\textbf{1}_{i > \frac{m}{2}} \in \Omega_{\emptyset}$. For $z \in X_{0},Y_{0}$, $\P_{z}[\tEsc[\emptyset] = 1] \approx \frac{1}{6}$, and if $X_{1}, Y_{1} \notin \Omega_{\emptyset}$, they must be in different partitions. Thus, inequality \eqref{IneqContractionAssumption} cannot be satisfied for any $\beta \ll \frac{1}{12}$. By standard arguments concerning the contraction of simple random walk on the hypercube (see Example 8 of \cite{Olli10}), inequality \eqref{IneqContractionAssumption}  cannot be satisfied for any $\alpha \gg \frac{1}{m}$. These constants do not satisfy the conditions of Theorem \ref{ThmContCondHighDim}. \par
In this example (and many others), the constants can be substantially improved by taking a trace of this chain. For $0 \leq k \leq \ell-1$ and a subset $z \subset [m]$, define 
\be 
\Omega_{z}^{(k)} = \{ x \in \Omega \, : \, \forall i \notin z, \, x_{i} \leq \ell-1-k; \, \forall i \in z, \, x_{i} \geq \ell + k\}.
\ee 
Set $\Omega^{(k)} = \cup_{z \subset [m]} \Omega^{(k)}_{z}$. Let $\tilde{K}$ be the transition kernel of the trace of $K$ on $\Omega^{(k)}$. 
We show that any fixed $\ell, k \geq 2$, inequality  \eqref{IneqContractionAssumption} is satisfied for the kernel $\tilde{K}$ with constants  
\be 
\alpha &= (1 - \frac{1}{2 m})(1 + o(1)),
\beta = o(1)
\ee 
for $C > 6$ as $m$ goes to infinity (see Equation \eqref{EqContAssumptionToyHighDimConc} in Appendix C). Here $C$ is the constant appearing in Equation \eqref{eqn:targ}. We also prove $\pi(\Omega^{(k)}) = 1 - o(1)$ for $C > 6$ as $m$ goes to infinity (see Equation \eqref{IneqHighDimToyPreCoup2}). As shown below, these constants \emph{are} good enough to be useful.

We now show how Theorem \ref{ThmContCondHighDim} can be applied to our example as $m$ goes to infinity for fixed $\ell \geq 3$, $1 \leq k < \ell -1$, and for $6 < C < \infty$. We show that for $m$ sufficiently large this example satisfies the conditions of Theorem \ref{ThmContCondHighDim} with constants

\be[IneqBigListOfConstantsToyHighDimEx]
\alpha &= 1 - \frac{1}{2m}, \, \beta = \frac{1}{m^{3}}, \, \gamma = \frac{1}{4}, \, D_{\max} = m, \\
a_{1} = 16 a_{2} &= 2 \frac{\E[\tEsc[\emptyset] | X_{0} = (0,0,\ldots,0)] \log_{2}(e)}{\varphi_{\emptyset} m } > 1, \\
\delta_{1} = \delta_{2} &= \frac{1}{2}, \, \overline{\varphi} \leq 2m \log(m), \, T = 8 m \log(8m). \\
\ee 
From Theorem \ref{ThmContCondHighDim}, we conclude 
\be 
\tmix = O(m \log(m) \E[\tEsc[\emptyset] | X_{0} = (0,0,\ldots,0)]).
\ee 
This is a reasonable estimate. Indeed, by considering the escape time from any part of the partition, it can be verified that $ m \, \E[\tEsc[\emptyset] | X_{0} = (0,0,\ldots,0)] = O(\tmix)$. Thus, our estimate is off by at most a factor of $\log(m)$. By inequality \eqref{IneqExpectedHittingTimeLowerBound}, $m^{C-2} = O(\E[\tEsc[\emptyset] | X_{0} = (0,0,\ldots,0)])$, and so this factor of $\log(m)$ is small relative to the mixing time.
\end{example}

\section{Applications} \label{SecAppl}

In this section, we apply our results to two Markov chains, illustrating some situations under which our bounds work well.

\subsection{Pince-Nez Graph} \label{SubsecPinceNez}
We carefully study the symmetric random walk on $2m$-vertex ``pince-nez" graph mentioned in the Introduction. Fix $m \in \mathbb{N}$ and define $\Omega_{1} = [m]$, $\Omega_{2} = \{m+1, m+2,\ldots,2m\}$ and $\Omega = \Omega_{1} \cup \Omega_{2}$. This was also the partition considered in \cite{JSTV04}. We will show that its mixing time (and thus its relaxation time) is $O(m^{2})$. 

 For $x \neq y$ and $x, y \in \Omega_{1}$ or $x, y \in \Omega_{2}$, we set $K(x,y) = \frac{1}{6}$ if $|x - y | = 1$ or $\{ x, y \} \in \{ \{1,m\}, \{m+1, 2m \} \}$. We also set $K(1,m+1) = K(m+1,1) = \frac{1}{6}$. For all other $x \neq y$, we set $K(x,y) = 0$. To complete the definition of the kernel, set $K(x,x) = 1 - \sum_{y \in \Omega} K(x,y)$. \par 

We claim that $\tmix = O(m^{2})$. We will prove this using Lemma \ref{LemmaBasicMixing}. By Example 10.20 of \cite{LPW09},
\be \label{IneqPinceMix}
\varphi_{\max} = O(m^{2}).
\ee 
For all $T \in \mathbb{N}$, 
\be 
\max_{x \in \Omega_{2} }\P_{x}[\tau_{\{ 1 \} } > T+1] &\leq (1 - \frac{1}{6} \min_{x \in \Omega_{2}} \P_{x}[\tau_{\{m+1\}} \leq T]) \\
&= \frac{5}{6} + \frac{1}{6}  \max_{x \in \Omega_{2} }\P_{x}[\tau_{ \{ m+1 \}} > T] \leq \frac{5}{6} + \frac{1}{6} \max_{x \in \Omega_{2} } \frac{\E_{x}[\tau_{ \{ m+1 \}}]}{T}. \label{IneqPinceNezExp}
\ee 
By Example 10.20 of \cite{LPW09}, $\max_{x \in \Omega_{2}} \E_{x}[\tau_{ \{ m+1 \}}] = \max_{x \in \Omega_{1}} \E_{x}[\tau_{ \{ 1 \}}] = \Theta(m^{2})$, and so inequalities \eqref{IneqPinceNezExp} and \eqref{eqn:subexp} imply $\max_{x \in \Omega} \E_{x}[\tau_{ \{ 1 \} }] = O(m^{2})$. By the symmetry of the problem,
for all $T \in \mathbb{N}$
\be \label{IneqPincenezSymmetry}
\E_{1}[\kappa_{1}(T)] \geq \frac{T+1}{2}.
\ee 
This implies that, for all $T \in \mathbb{N}$
\be 
\min_{x \in \Omega} \E_{x}[\kappa_{1}(T)] \geq \frac{1}{2} \min_{x \in \Omega} \E_{x}[(T - \tau_{ \{1 \}})] = \frac{T}{2} - \Theta(m^{2}). 
\ee 
Applying this bound and Markov's inequality, we conclude that for all $C_{1} > 0$, there exists $C_{2} > 0$ so that 
\be \label{IneqPinceHit1}
\max_{x \in \Omega} \P[\kappa_{1}(C_{2} m^{2}) \leq C_{1} m^{2}] \leq \frac{1}{8}.
\ee 
By the symmetry of the problem, 
\be \label{IneqPinceHit2}
\max_{x \in \Omega} \P[\kappa_{2}(C_{2} m^{2}) \leq C_{1} m^{2}] \leq \frac{1}{8}
\ee 
for the same $C_{1}, C_{2}$. The conclusion that $\tmix = O(m^{2})$ follows immediately from applying  Lemma \ref{LemmaBasicMixing} with bounds \eqref{IneqPinceMix}, \eqref{IneqPinceHit1}, \eqref{IneqPinceHit2}.

This bound on the mixing time immediately implies that the relaxation time of the walk is $O(m^{2})$ as well. It is straightforward to check (\textit{e.g.}, by the central limit theorem) that the mixing time $\tmix$ of the walk satisfies $m^{2} = O(\tmix)$. From the symmetry of the problem and the fact that the relaxation time of simple random walk on the cycle is $\Theta(m^{2})$, the relaxation time of this walk is also at least on the order of $m^{2}$.

This example is, of course, simple enough to be analyzed directly. We include it to illustrate the fact that we can obtain qualitatively better bounds than previous decomposition bounds.

\subsection{Toy Version of the Kinetically Constrained Ising Process} \label{ExKcipToy}

Our interest in decomposition bounds was motivated by our study of the following kinetically constrained Ising process, which originated in \cite{AnFr84}:

\begin{example} [Kinetically Constrained Ising Processes (KCIP) and Partitions] \label{ExKcipCarefulStatement}

Fix a constant $c> 0$, a graph $G = (V,E)$ with $|V| = m$ vertices and a function $N \, : \, G \mapsto 2^{G}$. A {KCIP} associated with graph $G$ and density $p = \frac{c}{m}$ is a Markov chain $\{X_{t} \}_{t \geq 0}$ on $\Omega = \{0,1\}^{G} \backslash \{ (0,0,\ldots,0)\}$ that has a transition kernel defined by the following algorithm for constructing $X_{t+1}$ from $X_{t}$:
\begin{enumerate}
\item Choose a vertex $v \in V$ and a number $\lambda \in [0,1]$ uniformly at random.
\item \textbf{If} there exists $u \in V$ such that $u \in N(V)$ and $X_{t}[u] = 1$, set $X_{t+1}[v] = 1$ if $\lambda \leq p$ and set $X_{t+1}[v] = 0$ if $\lambda > p$. Set $X_{t+1}[w] = X_{t}[w]$ for all $w \neq v$. 
\item \textbf{Otherwise,} if $X_{t}[u] = 0$ for all $u \in V$ such that $(u, v) \in E$, set $X_{t+1}[w] = X_{t}[w]$ for all $w \in V$.  
\end{enumerate}

It is natural to try to analyze this Markov chain by partitioning $\Omega$ according to the number of non-adjacent particles, fixing $n$ and defining for $1 \leq k \leq n$
\be \label{EqKcipPartition}
\Omega_k = \big \{ X \in \{0,1\}^G \, : \, \sum_{v \in V} X[v] = k, \sum_{ u \in G} \sum_{v \in N(u)} X[u]X[v] = 0 \big \}
\ee 
as well as the `remainder' $\Omega' = \Omega \backslash \cup_{k=1}^{n} \Omega_{k}$. Here $n$ is taken to be $O(1)$. This is because, in the regime $p = {c \over m}$, the stationary distribution is concentrated around $\Omega_k$ for $k = O(1)$.

In \cite{PiSm15}, we show that for $G = \mathbb{Z}_{L}^{3}$ the $3$-dimensional lattice on $m = L^{3}$ vertices and $N(v) = \{ u \, : \, (u,v) \in E\}$ the usual neighborhood of a vertex, the mixing time of the restricted kernels $K_{i}$ of this chain are $O(m^{\frac{8}{3}})$ and the mixing time of the projected kernel $\overline{K}$ is $O(m^{3} \log(m))$. Using Lemma \ref{LemmaBasicMixing} of this paper, along with calculations similar to those in Theorem \ref{ThmDecompDrift}, Lemma \ref{LemmaSimpSupBdBad2}, and Corollary \ref{CorSimpSupBdBad2}, we obtain a mixing bound of $O(\max( m^{\frac{8}{3}}, m^{3} \log(m)) )$ for the KCIP chain on $\mathbb{Z}_{L}^{3}$. 
\end{example}

We now give a toy version of this process and explain why the methods in this note improve upon existing bounds. Our toy process has a laziness parameter $d \geq 1$ and a size parameter $m \in \mathbb{N}$; for fixed $d$, we consider the asymptotics of the mixing time as $m$ goes to infinity. For each $m \in \mathbb{N}$, the state space of the model is $\Omega = \{ (i,j) \, : \, i \in [m], j \in \{ 1, 2, 3 \} \}$ and we consider the partition $\Omega_{i} = \{ (i,1), (i,2), (i,3) \}$, $1 \leq i \leq m$, of $\Omega$. The transition kernel $K$ is given by:
\be 
K( (i,1), (i+1, 1) ) &= K( (i,1), (i,2) ) = \frac{1}{6}, \, K( (i,1), (i-1,1) ) = \frac{1}{3}, \\
 K( (i,2), (i,1) ) &= K( (i,2),(i,3)) = K((i,3),(i,2)) =  \frac{1}{6m^{d}},
\ee 
where the first and third expressions assume that $i < m$ and $i > 1$ respectively. For all other $(i_{1},j_{1}) \neq (i_{2},j_{2})$, $K((i_{1},j_{1}), (i_{2}, j_{2})) = 0$. Finally, we set $K(x,x) = 1 - \sum_{y \neq x} K(x,y)$. This completes the definition of the transition matrix. It is immediate that
\be \label{IneqKcipToyMix}
\varphi_{i} = \Theta(m^{d}).
\ee 
Denote by $\{X_{t} = (X_{t}[1], X_{t}[2])\}_{t \in \mathbb{N}}$ a Markov chain driven by $K$. We claim that, for any $\epsilon > 0$ and $m > N(\epsilon)$ sufficiently large,
\be \label{IneqKcipToyDrift}
\max_{x \in \Omega} \E[e^{\frac{1}{2}X_{t+ \epsilon m^{1+d}}[1]} | X_{t} = x] = O(1).
\ee 
 To prove this, let $\Omega_{\mathrm{lower}} \equiv \{ (i,1) \, : \, 1 \leq i \leq m \}$ and let  $\{ Y_{s} \}_{s \in \mathbb{N}}$ be the trace of $\{ X_{s} \}_{s \in \mathbb{N}}$ on $\Omega_{\mathrm{lower}}$. Under the identification $(i,1) \mapsto i$ of $\Omega_{\mathrm{lower}}$ with $[m]$, $\{ Y_{s} \}_{s \in \mathbb{N}}$ has transition kernel
\be 
K_{\mathrm{lower}}(i,i+1) = \frac{1}{6}, \, K_{\mathrm{lower}}(i,i-1) =
\frac{1}{3}, \, K_{\mathrm{lower}}(i,i) = \frac{1}{2} 
\ee 
for $i \neq \{1,m\}$. We can directly compute that $\{Y_{s} \}_{s \in \mathbb{N}}$ satisfies
\be 
\E[e^{\frac{1}{2}Y_{s+1}} | Y_{s}] \leq 0.98 e^{\frac{1}{2} Y_{s}} + 0.25.
\ee 
Let $\kappa_{\mathrm{lower}}(T) = | \{ t \leq s \leq t + T \, : \,  X_{s} \in \Omega_{\mathrm{lower}} \}$. By Corollary 2.11 of \cite{Paul14}, $\kappa_{\mathrm{lower}}(\epsilon m^{1+d}) = \Theta(\epsilon m)$ as $m$ goes to infinity. This implies
\be 
\E[e^{\frac{1}{2} X_{t + \epsilon m^{1+d}}[1]} | X_{t}] &= \E[e^{\frac{1}{2} Y_{\kappa(\epsilon m^{1+d})}} | X_{t} ] \\
&\leq (0.98 + o(1))^{\Theta(\epsilon m)} e^{\frac{1}{2} X_{t}} + \frac{0.25}{1-0.98},
\ee 
which proves inequality \eqref{IneqKcipToyDrift}. \par

Fix $C>0$. We consider the trace of $\{ X_{t} \}_{t \in \mathbb{N}}$ onto $\Omega' = \cup_{i=1}^{C} \Omega_{i}$. As each part $\Omega_{i}$ of the partition of $\Omega$ has three states, it is possible to check by direct computation that with the choice $c = \frac{1}{10}$ all of the constants $\delta,\epsilon,D$ in the conditions of Corollary \ref{CorSimpSupBdBad2} and Lemma \ref{LemmaSimpSupBdBad2} are $\Theta(1)$ as $m$ goes to infinity for this trace (where the implied constants depend on $C$). Thus, by Corollary \ref{CorSimpSupBdBad2} and Lemma \ref{LemmaSimpSupBdBad2}, the mixing time $\tmix'$ of the trace of $\{ X_{t} \}_{t \in \mathbb{N}}$ onto $\Omega' = \cup_{i=1}^{C} \Omega_{i}$ is at most $\tmix' = O(m^{d})$. By Theorem \ref{ThmDecompDrift} and inequality \eqref{IneqKcipToyDrift}, this implies that the mixing time $\tmix$ of $\{ X_{t} \}_{t \in \mathbb{N}}$ must be at most $\tmix = O(m^{1+d} \log(m))$. This also immediately implies that the relaxation time $\tau_{\mathrm{rel}} = O(m^{1+d} \log(m))$. \par
We compare this bound to the bounds achievable by \cite{JSTV04}. Recall that their projected chain $\overline{K}$ is given by Equation \eqref{EqDefProjChain}, while their restricted chains are given by
\be 
K_{i}(x,y) = K(x,y) \textbf{1}_{y \in \Omega_{i}}
\ee 
for $x \neq y$ and $K_{i}(x,x) = 1 - \sum_{y \in \Omega_{i}} K(x,y)$. In this example, the $m$ restricted chains have three points and the projected chain is an $m$-state, $(1-\Theta(m^{-d}))$-lazy birth and death chain corresponding to random walk on the path with constant drift. \par

These calculations let us compute the bound given by Theorem 1 of \cite{JSTV04}. Using the trivial bound $\gamma \leq 1$ on the correction term, these bounds imply that this walk has relaxation time $\tau_{\mathrm{rel}} = O(m^{2d+1})$, which implies that $\tmix = O(m^{2d+1} \log(m))$. For $d$ large, the difference between our bounds and those in \cite{JSTV04} is substantial. The discrepancy between the mixing bounds stems from the fact that our two main bounds, on the number of steps required to establish the drift condition inequality \eqref{IneqKcipToyDrift} and on the mixing time \eqref{IneqKcipToyMix} of the projected chains, are \textit{added} to obtain our final bound on the mixing time, while the corresponding bounds must be \textit{multiplied} to obtain the final bound in \cite{JSTV04}. We emphasize that the discrepancy between our bounds remains large even if we restrict our attention to the mixing time of the trace of  $\{ X_{t} \}_{t \in \mathbb{N}}$ onto $\Omega' = \cup_{i=1}^{C} \Omega_{i}$. \footnote{The authors of \cite{JSTV04} are aware of these issues, and give various refinements to their main inequality (see Equations 2 and 21 of \cite{JSTV04}). However, this refinement does not make a qualitative difference to the analysis of this example.  } \par

\subsubsection{Discussion of Other Interacting Particle Systems} \label{ExGenInt} 

The difficulties illustrated in the toy example in Section \ref{ExKcipToy} apply to more realistic interacting particle systems, including the KCIP models defined in Example \ref{ExKcipCarefulStatement}. We briefly sketch the problem, as complete proofs would be long. Fix a sequence of connected graphs $\{ G_{n} \}_{n \in \mathbb{N}}$ with $|G_{n}|=n$ and maximum degree $deg(G_{n}) \leq d$ for fixed $ 0 < d < \infty$; also fix a constant $0 < c < \infty$ as in that example.

Let $K$ be the transition kernel given in Example \ref{ExKcipCarefulStatement}, and let $\{K_{i}\}_{i=1}^{n}$, $\overline{K}$ be the projected and restricted kernels associated with $K$ and the partition given in Equation \eqref{EqKcipPartition}. Let $1-\lambda_{i}$, $1 - \overline{\lambda}$ and $1 - \lambda$ be the spectral gaps of $K_{i}$, $\overline{K}$ and $K$ respectively, and let $\pi$ be the stationary distribution of $K$. The calculation in Lemma 4.1 of \cite{PiSm15} implies that $\overline{K}(1,1) = 1 - O(n^{-3})$\footnote{Lemma 4.1 makes the assumption that the associated graph is undirected and triangle-free. However, the calculation goes through with only minor modification and gives the stated conclusion.} and $\pi(1) = \Theta(1)$, which in turn implies $ \overline{\lambda} = 1 - O(n^{-3})$. Next, note that $K_{1}(x,x) = 1 - O(n^{-2})$ for all $x \in \Omega_{1}$. In the special case that 
\be \label{EqSpecialCase}
\mathcal{N}(v) = \{ u \, : \, (u,v) \in E\},
\ee we have $\lambda_{1} = 1 - O(n^{-2} \max(n^{-1},( 1-\lambda_{n})))$, where $1 - \lambda_{n}$ is the spectral gap associated with the random walk on $G_{n}$\footnote{This follows from the fact that $K_{1}(x,x) = 1 - O(n^{-2})$ for all $x \in \Omega_{1}$, combined with Lemma 4.1 of \cite{PiSm15} and a standard comparison argument (see, \textit{e.g.}, \cite{DiSa93b})}. Finally, the correction term $\gamma$ defined in Equation 21 of \cite{JSTV04},
\be 
\gamma \equiv \max_{1 \leq i \leq m} \max_{x \in \Omega_{i}} \sum_{y \notin \Omega_{i}} K(x,y),
\ee 
satisfies $\gamma^{-1} = O(n^{2})$. 

Thus, the best possible bound obtainable by Theorem 1 of \cite{JSTV04} in the special case \eqref{EqSpecialCase} is $\frac{1}{1 - \lambda} = O(n^{3} \frac{1}{1 - \lambda_{n}})$. Even this bound is rather optimistic (for more complicated reasons, it is also not possible to obtain a bound better than $\frac{1}{1 - \lambda} = O(n^{4} )$ using a decomposition of the form \eqref{EqKcipPartition}), but is already quite far from the truth. For the main example studied in \cite{PiSm15},  the correct answer is at most $\frac{1}{1 - \lambda} = O(n^{3} \log(n))$, while this optimistic bound would give  $\frac{1}{1 - \lambda} = O(n^{\frac{11}{3}})$. Similar difficulties occur outside of the special case \eqref{EqSpecialCase} (see, \textit{e.g.}, \cite{chleboun2014relaxation} for the correct relaxation time for a KCIP that does not satisfy \eqref{EqSpecialCase} ).

\section*{Acknowledgements}
NSP is partially supported by an ONR grants. AMS is partially supported by an NSERC grant.
\bibliographystyle{alpha}
\bibliography{BLDBib}

\begin{thebibliography}{DGJM06}

\bibitem[AF84]{AnFr84}
H.C. Andersen and G.H. Frederickson.
\newblock Kinetic {I}sing models of the glass transition.
\newblock {\em Phys. Rev. Lett.}, 53(13):1244--1247, 1984.

\bibitem[BD97]{BuDe97}
Russ Bubley and Martin Dyer.
\newblock Path coupling: A technique for proving rapid mixing in {M}arkov
  chains.
\newblock {\em Proceedings of the 38th IEEE Symposium on Foundations of
  Computer Science (FOCS)}, pages 223--231, 1997.

\bibitem[BSZ11]{BSZ10}
N.~Berestycki, O.~Schramm, and O.~Zeitouni.
\newblock Mixing times for random k-cycles and coalescence-fragmentation
  chains.
\newblock {\em Annals of Probability}, 39(5):1815--1843, 2011.

\bibitem[CFM14]{chleboun2014relaxation}
Paul Chleboun, Alessandra Faggionato, and Fabio Martinelli.
\newblock Relaxation to equilibrium of generalized east processes on $z^{d} $:
  Renormalization group analysis and energy-entropy competition.
\newblock {\em arXiv preprint arXiv:1404.7257}, 2014.

\bibitem[DGJM06]{DGJM06}
Martin Dyer, Leslie Goldberg, Mark Jerrum, and Russell Martin.
\newblock {M}arkov chain comparison.
\newblock {\em Probability Surveys}, 3:89--111, 2006.

\bibitem[DLP10]{DLP10}
Jian Ding, Eyal Lubetzky, and Yuval Peres.
\newblock Mixing time of critical {I}sing model on trees is polynomial in the
  height.
\newblock {\em Communications in Mathematical Physics}, (295):161--207, 2010.

\bibitem[DSC93]{DiSa93b}
Persi Diaconis and Laurent Saloff-Coste.
\newblock Comparison theorems for reversible {M}arkov chains.
\newblock {\em Annals of Applied Probability}, 3(3):696--730, 1993.

\bibitem[HLW06]{hoory2006expander}
Shlomo Hoory, Nathan Linial, and Avi Wigderson.
\newblock Expander graphs and their applications.
\newblock {\em Bulletin of the American Mathematical Society}, 43(4):439--561,
  2006.

\bibitem[JSTV04]{JSTV04}
Mark Jerrum, Jung-Bae Son, Prasad Tetali, and Eric Vigoda.
\newblock Elementary bounds on {P}oincar\'{e} and log-{S}obolev constants for
  decomposable {M}arkov chains.
\newblock {\em Annals of Applied Probability}, 14(4):1741--1765, 2004.

\bibitem[KO15]{kovchegov2015path}
Yevgeniy Kovchegov and Peter~T Otto.
\newblock Path coupling and aggregate path coupling.
\newblock {\em arXiv preprint arXiv:1501.03107}, 2015.

\bibitem[LPW09]{LPW09}
David Levin, Yuval Peres, and Elizabeth Wilmer.
\newblock {\em {M}arkov Chains and Mixing Times}.
\newblock American Mathematical Society, Providence, Rhode Island, 2009.

\bibitem[MR00]{MaRa00}
Russel Martin and Dana Randall.
\newblock Sampling adsorbing staircase walks using a new {M}arkov chain
  decomposition method.
\newblock {\em Proceedings of the 41st IEEE Symposium on Foundations of
  Computer Science (FOCS)}, pages 492--502, 2000.

\bibitem[MR02]{MaRa02}
Neil Madras and Dana Randall.
\newblock {M}arkov chain decomposition for convergence rate analysis.
\newblock {\em Annals of Applied Probability}, 12:581--606, 2002.

\bibitem[MR06]{MaRa06}
Russel Martin and Dana Randall.
\newblock Disjoint decompositions of {M}arkov chains and sampling circuits in
  {C}ayley graphs.
\newblock {\em Combinatorics, Probability and Computing}, 15:411--448, 2006.

\bibitem[MY09]{MaYu09}
Neil Madras and Wai~Kong Yuen.
\newblock Spectral gaps of random walk {M}etropolis chains.
\newblock {\em Far East Journal of Theoretical Statistics}, 27:157--191, 2009.

\bibitem[Oli12]{Oliv12b}
Roberto Oliveira.
\newblock Mixing and hitting times for finite {M}arkov chains.
\newblock {\em Electronic Journal of Probability}, 17(70):1--12, 2012.

\bibitem[Oll10]{Olli10}
Yann Ollivier.
\newblock A survey of {R}icci curvature for metric spaces and {M}arkov chains.
\newblock {\em Adv. Stud. Pure Math.}, 57, 2010.

\bibitem[Pau15]{Paul14}
Daniel Paulin.
\newblock Concentration inequalities for {M}arkov chains by {M}arton couplings
  and spectral methods.
\newblock {\em arXiv preprint: arXiv:1212.2015}, 2015.

\bibitem[PS13]{PeSo13}
Yuval Peres and Perla Sousi.
\newblock Mixing times are hitting times of large sets.
\newblock {\em Journal of Theoretical Probability}, 9:459--510, 2013.

\bibitem[PS16]{PiSm15}
Natesh~S. Pillai and Aaron Smith.
\newblock Mixing times for a constrained {I}sing process on the torus at low
  density.
\newblock {\em Annals of Probability (Forthcoming)}, 2016.

\bibitem[Ros95]{Rose95}
Jeffrey Rosenthal.
\newblock Minorization conditions and convergence rates for {M}arkov chain
  {M}onte {C}arlo.
\newblock {\em Journal of the American Statistical Association}, 90:558--566,
  1995.

\bibitem[Sch02]{Schw02}
Jason Schweinsberg.
\newblock An ${O}(n^{2})$ bound for the relaxation time of a {M}arkov chain on
  cladograms.
\newblock {\em Random Structures and Algorithms}, 20(1):59--70, 2002.

\bibitem[Tie98]{tierney1998note}
Luke Tierney.
\newblock A note on {M}etropolis-{H}astings kernels for general state spaces.
\newblock {\em Annals of Applied Probability}, pages 1--9, 1998.

\bibitem[Zha15]{zhang2015multi}
Wei Zhang.
\newblock Asymptotic analysis of multiscale {M}arkov chain.
\newblock {\em arXiv preprint arXiv:1512.08944}, 2015.

\end{thebibliography}

\newpage

\section*{Appendix A}

Let $\{ X_{t} \}_{t \geq 0}$ be a $\frac{1}{2}$-lazy, irreducible, reversible Markov chain on a finite state space $\Omega$ with transition kernel $P$, and let $\{ Y_{t} \}_{t \geq 0}$ be the Markov chain with transition kernel $Q = \frac{1}{2}(\mathrm{Id} + P)$ and initial point $Y_{0} = X_{0} = x$.
Fix a subset $A \subset \Omega$ and let $\tau_{A} = \min \{ t > 0 \, : \, X_{t} \in A \}$ and $\tau_{A}' = \min \{ t > 0 \, : \, Y_{t} \in A \}$. Finally, let $\tmix$ denote the mixing time of $X_t$.
Theorem 1.1 of \cite{PeSo13} says that there exist universal constants $d_\alpha, d'_\alpha$ such that 
\be \label{eqn:PSact}
 d'_\alpha \max_{z,A: \pi(A) \geq \alpha} \E_z[\tau'_A] \leq \tmix \leq  d_\alpha \max_{z,A: \pi(A) \geq \alpha} \E_z[\tau_{A}'],
 \ee
 whereas our formulation in \eqref{eqn:PS13eqn} bounds $\tmix$ in terms of $\tau_A$. The following simple result will be used to show that our inequality \eqref{eqn:PS13eqn} is equivalent to Theorem 1.1 of \cite{PeSo13}.
%The following is enough to reduce inequality \eqref{eqn:PS13eqn} to Theorem 1.1 of \cite{PeSo13}:

\begin{lemma} \label{LemmaVersionOfPeso}
For any $A \subset \Omega$, there exists a universal constant $C$ that does not depend  on $P$, $A$ or $\Omega$ such that
\be \label{IneqEquivOfHittingTimes}
\E[\tau_{A}] \leq \E[\tau_{A}'] \leq 8 \E[\tau_{A}] + C.
\ee 
\end{lemma}
\begin{proof}
The lower bound in inequality \eqref{IneqEquivOfHittingTimes} is trivial. To prove the upper bound, we introduce a version of $\{Y_{t} \}_{t \geq 0}$ on an augmented state space as follows. Let $\{ A_{t} \}_{t \geq 0}$ be an i.i.d. Bernoulli($\frac{1}{2}$) sequence, and then construct $\{ Y_{t} \}_{t \geq 0}$ by drawing from the kernel:
\be 
\P[Y_{t+1} \in \cdot |  A_{t} = 1, Y_{t}] &= P(Y_{t},\cdot) \\
\P[Y_{t+1} \in \cdot |  A_{t} = 0, Y_{t}] &= \delta_{Y_{t}}(\cdot). 
\ee 
This construction of $\{Y_{t} \}_{t \geq 0}$ has transition kernel $Q = \frac{1}{2}(\mathrm{Id} + P)$. Let $N_{t} = \sum_{s=0}^{t-1} A_{s}$ and let $M_{t} = \min \{ s > 0 \, : \, N_{s} =t \}$. We then construct $\{X_{t} \}_{t \geq 0}$ by setting
\be 
X_{t} = Y_{M_{t}};
\ee 
this construction of $\{ X_{t}\}_{t \geq 0}$ yields a Markov chain with transition matrix $P$. We then have, for all $0 \leq a < 8$ and all $t \in \mathbb{N}$ that 
\be 
\{ \tau_{A}' > 8t + a \}  \subset \{ \tau_{A} > t \} \cup  \{ N_{8t} \leq t \},
\ee 
and so
\be 
\P[\tau_{A}' > 8t + a] \leq \P[\tau_{A} > t] + \P[N_{8t} \leq t].
\ee 
Summing this over $t$ and $a$, we have 
\be 
\E[\tau_{A}'] &= \sum_{a=0}^{7} \sum_{t=0}^{\infty} \P[\tau_{A}' > 8t + a] \\
&\leq \sum_{a=0}^{7} \sum_{t=0}^{\infty} (\P[\tau_{A} > t] + \P[N_{8t} \leq t]) \\
&\leq 8 \E[\tau_{A}] + 8 \sum_{t=0}^{\infty} \P[N_{8t} \leq t]. 
\ee 
Since $\sum_{t=0}^{\infty} \P[N_{8t} \leq t] < \infty$ by Hoeffding's inequality, the upper bound in inequality \eqref{IneqEquivOfHittingTimes} follows. 
\end{proof}
Since $\max_{z, A \subset \Omega} \mathbb{E}(\tau_A) \geq 1$, from
 \eqref{IneqEquivOfHittingTimes} it immediately follows that 
 \be
 \max_{z,A} \E[\tau_{A}] \leq  \max_{z,A} \E[\tau_{A}'] \leq  C' \max_{z,A} \E[\tau_{A}] 
 \ee
 for some universal constant $C'>0$. This in turn combined with \eqref{eqn:PSact} immediately implies that
 \be
 c'_\alpha \max_{z,A: \pi(A) \geq \alpha} \E_z[\tau_A] \leq \tmix \leq  c_\alpha \max_{z,A: \pi(A) \geq \alpha} \E_z[\tau_{A}]
 \ee
 for some universal constants $c_\alpha, c'_\alpha$ as claimed in 
 Equation \eqref{eqn:PS13eqn}.

We also prove inequality \eqref{eqn:subexp2} from the introduction:

\begin{proof}[Proof of inequality \eqref{eqn:subexp2}]
Let $T = \max \{ t \in \mathbb{N}\, : \, \max_{x \in \Omega} \P_{x}[\tau_{A} > t] > e^{-1} \}$. Then Markov's inequality gives
\be 
\max_{x \in \Omega} \E_{x}[\tau_{A}] \geq T  \max_{x \in \Omega} \P_{x}[\tau_{A} > T] \geq e^{-1} T.
\ee 

Combining this with inequality \eqref{eqn:subexp}, we have for all $x \in \Omega$ and $k \in \mathbb{N}$,
\be \label{IneqEndOfShortProof}
\P_{x}[\tau_{A} > k e \max_{y \in \Omega}\E_{y}[\tau_{A}]] \leq \P_{x}[\tau_{A} > kT] \leq e^{-k}. 
\ee 
This immediately implies the desired inequality.
\end{proof}

\section*{Appendix B}
We prove a series of standard technical lemmas, leading to the proof of Theorem \ref{ThmDecompDrift}:

\begin{lemma}[Drift Implies Concentration] \label{LemmaDriftConc}
Let $K$ be the transition kernel of a Markov chain satisfying the conditions given in Theorem \ref{ThmDecompDrift}. Then
\be 
\pi(\mathcal{L}(M)) \geq 1 - \frac{b}{aM} \geq \frac{3}{4}.
\ee  
\end{lemma}
\begin{proof}
Let $\{ X_{t} \}_{t \in \mathbb{N}}$ be a Markov chain with transition kernel $K$, started at stationarity, \textit{i.e.}, $X_{0} \sim \pi$. Since $X_{k}$ then has distribution $\pi$ as well,
\be 
\pi(V) &= \E[\E[V(X_{k}) | X_{0}] ] \\
&\leq \E[(1-a)V(X_{0}) + b] \\
&= (1-a) \pi(V) + b.
\ee 
Thus, $\pi(V) \leq \frac{b}{a}$ 
and so by Markov's inequality,
\be 
\pi(\mathcal{L}(M)) \geq 1 - \frac{b}{aM}.
\ee 
Since $4a/b < M$, $\pi(\mathcal{L}(M)) \geq 1 - \frac{b}{aM} \geq 3/4$ and the proof is finished.
\end{proof}

\begin{lemma} \label{LemmaTechLem1}
Fix $0 < \beta \leq 1$ and $0 \leq \gamma < \infty$. Consider a stochastic process $\{ X_{t} \}_{t \in \mathbb{N}}$ with associated filtration $\mathcal{F}_{t}$ that satisfies the drift condition 
\be
\E[X_{s+1} \vert \mathcal{F}_{s}] \leq (1 - \beta) X_{s} + \gamma  
\ee
for all $s \in \mathbb{N}$. Let $Z_{0}, Z_{1}, \ldots$ be an i.i.d. sequence of random variables with geometric distribution and mean $\frac{2}{\beta}$. Then, if $X_{0} \leq \frac{4 \gamma}{\beta}$, we have for all $C, T \in \mathbb{N}$ that
\be 
\P \big[\sum_{s=0}^{T} \mathbf{1}_{X_{s} < \frac{4 \gamma}{\beta}} < C \big] \leq \P \big[\sum_{i=1}^{C} Z_{i} > T \big].
\ee 
\end{lemma}

\begin{proof}
Assume $X_{0} < \frac{4 \gamma}{\beta}$ and let $\tRe = \min \{ t > 0 \, : \, X_{t} < \frac{4 \gamma}{\beta} \}$. Then for all $s \geq 2$,
\be 
\E[X_{s} \textbf{1}_{\tRe \geq s}] &= \E[ \E[ X_{s} \textbf{1}_{\tRe \geq s} | \mathcal{F}_{s-1}]] \\
&\leq \E[ ((1 - \beta)X_{s-1} + \gamma)\textbf{1}_{\tRe \geq s}] \\
&= \E[ ((1 - \frac{\beta}{2})X_{s-1} - \frac{\beta}{2} X_{s-1} + \gamma)\textbf{1}_{\tRe \geq s}]\\
&\leq \E[ ((1 - \frac{\beta}{2})X_{s-1} - \frac{\beta}{2} \frac{4 \gamma}{\beta} + \gamma)\textbf{1}_{\tRe \geq s}] \\
&\leq (1 - \frac{\beta}{2}) \E[X_{s-1} \textbf{1}_{\tRe \geq s}].
\ee 
Iterating this inequality and noting that $\E[X_{1} \textbf{1}_{\tRe \geq s}] \leq (1 - \beta) \frac{4 \gamma}{\beta} + \gamma \leq \big(1 - \frac{\beta}{2} \big) \frac{4 \gamma}{\beta}$, we have
\be 
\E[X_{s} \textbf{1}_{\tRe \geq s}] \leq  \big(1 - \frac{\beta}{2} \big)^{s} \frac{4 \gamma}{\beta}, \\
\ee 
and so
\be  \label{IneqExpDom0}
\P[\tRe > s ] &\leq \P[X_{s} \textbf{1}_{\tRe \geq s} > \frac{4 \gamma}{\beta}]  \leq \big(1 - \frac{\beta}{2} \big)^{s}.
\ee
Define $t_{0} = 0$ and $t_{i+1} = \min \{ s > t_{i} \, : \, X_{s} \leq \frac{4 \gamma}{\beta} \}$. The fact that inequality \eqref{IneqExpDom0} holds uniformly in $X_{0}$ implies that
\be \label{IneqExpDom}
\P[t_{i+1} - t_{i} > s | \{ t_{j} \}_{j \leq i} ] \leq  \big(1 - \frac{\beta}{2} \big)^{s}.
\ee  
Then
\be \label{IneqExpDom2}
\P[\sum_{s=0}^{T} \textbf{1}_{X_{s} < \frac{4 \gamma}{\beta}} < C] \leq \P[t_{C} > T] = \P[\sum_{i=1}^{C} (t_{i} - t_{i-1}) > T]. \\
\ee 
Inequality \eqref{IneqExpDom} implies that the distribution of $(t_{i+1} - t_{i})$ is (conditionally on $\{ t_{j} \}_{j \leq i }$) stochastically dominated by a geometric distribution with mean $\frac{2}{\beta}$; combining this with inequality \eqref{IneqExpDom2} completes the proof.
\end{proof}

We apply this to show that the set $\mathcal{L}(C)$ defined in equation \eqref{EqDefSmallSets} has moderately large occupation measure for $C$ sufficiently large:

\begin{cor} \label{CorMainDriftResult}
Let $\{X_{t}\}_{t \in \mathbb{N}}$ be a Markov chain satisfying the conditions given in Theorem \ref{ThmDecompDrift} and let $\mathcal{L}$ be as in equation \eqref{EqDefSmallSets}. For fixed $C_{1}> \frac{4b}{a}, 0 \leq C_{2} \leq \frac{a}{8}$  and any starting point $X_{0}$ and time $U \in \mathbb{N}$,
\be 
\P[\sum_{t=0}^{U} \mathbf{1}_{X_{t} \in \mathcal{L}(C_{1})} < C_{2} U] \leq   \vmax e^{- \frac{a}{2} \lfloor \frac{U}{2k} \rfloor} + e^{- \frac{a U}{32} } .
\ee 
\end{cor}
\begin{proof}

Let $\tau_{\mathrm{start}} = \min \{ t > 0 \, : \, X_{t} \in \mathcal{L}(C_{1}) \}$. By inequality \eqref{IneqDriftCondBasic}, we have for any $X_{t}$ satisfying $V(X_{t}) > C_{1} > \frac{4b}{a}$ that 
\be 
\E[V(X_{t+k}) | X_{t}] &\leq (1 - a) V(X_{t}) + b\\
&\leq (1 - \frac{a}{2}) V(X_{t}) - \frac{a}{2} V(X_{t}) + b \\
&\leq (1 - \frac{a}{2}) V(X_{t}) - b \leq (1 - \frac{a}{2}) V(X_{t}).
\ee 
By Markov's inequality and the trivial bound that $V_{t} \leq \vmax$ for all $t \in \mathbb{N}$, this implies 
\be
\P[\tau_{\mathrm{start}} > kt] &\leq \P[V_{k t} \textbf{1}_{\tau_{\mathrm{start}} > kt} > C_{1} ]\\
&\leq \vmax \big(1 - \frac{a}{2} \big)^{t}.  \label{IneqLemmaTechLem2Start}
\ee 

Fix $T \leq U \in \mathbb{N}$ and let  $\{ Z_{i}\}_{i \in \mathbb{N}}$ be an i.i.d. sequence of random variables with geometric distribution and mean $\frac{2}{\alpha}$ . By inequality \eqref{IneqLemmaTechLem2Start}, the Markov property and Lemma \ref{LemmaTechLem1},
\be 
\P[ & \sum_{t=0}^{ U} \textbf{1}_{X_{t} \in \mathcal{L}(C_{1})} > C_{2} U] \geq \P[\sum_{t=T}^{ U} \textbf{1}_{X_{t} \in \mathcal{L}(C_{1})} > C_{2}  U | \tau_{\mathrm{start}} < T] \P[\tau_{\mathrm{start}} < T] \\
& = \P[\tau_{\mathrm{start}} < T]  \sum_{t=0}^{T} \P[\sum_{s=t}^{ U} \textbf{1}_{X_{s} \in \mathcal{L}(C_{1})} > C_{2} U | \tau_{\mathrm{start}} = t] \P[\tau_{\mathrm{start}} = t | \tau_{\mathrm{start}} \leq T]\\
&\geq \big(1 - \vmax \big(1 - \frac{1}{2} a \big)^{\lfloor\frac{T}{k } \rfloor} \big)  \sum_{t=0}^{T} \P \big[\sum_{i=1}^{C_{2} U} Z_{i} \leq U - t \big] \P[\tau_{\mathrm{start}} = t | \tau_{\mathrm{start}} \leq T] \\
&\geq \big(1 - \vmax \big(1 - \frac{1}{2} a \big)^{\lfloor\frac{T}{k } \rfloor} \big)   \P \big[\sum_{i=1}^{C_{2} U} Z_{i} \leq U - T \big] .
\ee 
Choosing $T = \lfloor \frac{U}{2}  \rfloor$, we have for $C_{2} > \frac{a}{8}$ that
\be 
\P[  \sum_{t=0}^{ U} \textbf{1}_{X_{t} \in \mathcal{L}(C_{1})} > C_{2} U] \geq \big(1 - \vmax e^{- \frac{a}{2} \lfloor \frac{U}{2k} \rfloor}  \big) \big( 1 - e^{- \frac{a U}{32} } \big).
\ee  
The second part of the above inequality is standard concentration inequality for geometric random variables. 
%(see e.g. Theorem 1 of \cite{BrownS}). 
This completes the proof. \end{proof}

Theorem \ref{ThmDecompDrift} now follows immediately from Lemmas \ref{LemmaBasicMixing} and \ref{LemmaDriftConc} and Corollary \ref{CorMainDriftResult}:

\begin{proof} [Proof of Theorem \ref{ThmDecompDrift}]
We apply Lemma \ref{LemmaBasicMixing}, choosing in the notation of that lemma $[n] = 2$, $\Omega_{1} = \mathcal{L}(M)$, $\Omega_{2} = \Omega \backslash \Omega_{1}$, $I = \{ 1 \}$, $\alpha = \frac{3}{8}$, and $\beta = \frac{3}{4}$. Lemma \ref{LemmaDriftConc} implies that this choice of $I$ and $\beta$ satisfies the requirements of Lemma \ref{LemmaBasicMixing}. \\
Set $\gamma = \frac{1}{6}$. In the notation of Corollary \ref{CorMainDriftResult}, choosing  
\be
U > T \equiv \frac{4}{a} \max(\frac{16}{ c_{\gamma}'} \tmix', k \log(16 V_{\max}), 8 \log(16)),
\ee
$C_{1} = M$ and $C_{2} = \frac{8 \tmix'}{c_{\gamma}' T}$   gives
\be 
\P[\kappa_{1}(U) < \frac{8}{  c_{\gamma}'} \tmix'] \leq \frac{1}{8}.
\ee  
The result then follows immediately from Lemma \ref{LemmaBasicMixing}.
\end{proof}
\section*{Appendix C} 
We prove the claims made in Example \ref{RemContCondTrace2}, with the ultimate goal of applying Theorem \ref{ThmContCondHighDim}. 
\begin{lemma} [Coupling to One Point] \label{LemmaCoupOnePointSupp}
Consider a Markov chain $\{X_{t} \}_{t \geq 0}$ with transition kernel $K$ and stationary distribution $\pi$ on a finite state space $\Omega$ with privileged point $z \in \Omega$. Let $\tau = \min \{ t \geq 0 \, : \, X_{t} = z \}$. Assume that
\be 
\max_{x \in \Omega} \E[\tau | X_{0} = x] \leq T, \, \pi(z) \geq 1- \epsilon 
\ee 
for some $\epsilon < \frac{1}{4}$. Then the mixing time $\tmix$ of $K$ satisfies 
\be 
\tmix \leq \lceil e \, T \rceil \lceil \log(\frac{1 - 4 \epsilon}{4(1 - \epsilon)}) \rceil.
\ee 
\end{lemma}

\begin{proof}
Fix $x \in \Omega$. Let $\{X_{t} \}_{t \geq 0}$ be a copy of the Markov chain started at $X_{0} = x$, and let $\{Y_{t}\}_{t \geq 0}$ be a copy of the Markov chain started according to the stationary distribution, so that $Y_{0} \sim \pi$. Let $\tau_{\mathrm{coll}} = \min \{ t \geq 0 \, : \, X_{t} = Y_{t} \}$ be the collision time of $\{ X_{t} \}_{t \geq 0}, \{Y_{t}\}_{t \geq 0}$. We couple these two chains so that they move independently until time $\tau_{\mathrm{coll}}$ and satisfy $X_{s} = Y_{s}$ for all $s \geq \tau_{\mathrm{coll}}$. By inequality \eqref{eqn:subexp2}, we then have for all $t \in \mathbb{N}$,
\be 
\P[\tau_{\mathrm{coll}} \leq t] &\geq \P[\tau \leq t]\P[X_{\tau} = Y_{\tau} | \tau \leq t] \\
&\geq  (1 - e^{-\lfloor \frac{t}{e \, T} \rfloor})  (1-\epsilon).
\ee 
By the standard `coupling lemma' for Markov chains (see Prop 4.7 of \cite{LPW09}),
\be 
\tmix \leq \min \{t > 0 \, :\, \P[\tau_{\mathrm{coll}} \leq t] \geq \frac{3}{4} \} \leq \lceil e \, T \rceil \lceil \log(\frac{1 - 4 \epsilon}{4(1 - \epsilon)}) \rceil
\ee 
and the proof is finished.
\end{proof}
As mentioned before, we are interested in calculating the mixing time in Example \ref{RemContCondTrace2} with $k,\ell$ and $C > 6$ held constant and $m$ going to infinity.  Theorem \ref{ThmContCondHighDim} can also be applied when $C = C(m) \rightarrow \infty$ using similar (and indeed much easier) bounds; when $C = C(m) = o(1)$, we expect  Theorem \ref{ThmContCondHighDim} to become ineffective. In order to apply Theorem \ref{ThmContCondHighDim}, we must prove a contraction inequality of the form \eqref{IneqContractionAssumption} and also an occupation inequality of the form \eqref{IneqRegularityAssumptions2}.

We assume for the remainder of this section that $m \gg \ell$. We begin by proving the contraction estimate \eqref{IneqContractionAssumption}. Fix $x \in \Omega^{(k)}_{\emptyset}$, let $\{X_{t}\}_{t \geq 0}$ be a copy of the Markov chain evolving according to $\tilde{K}$ and started at $X_{0} = x$, and let $\tau_{\mathrm{centre}} = \min \{t > 0 \, : \, X_{t} = (0,0,\ldots,0) \}$. We begin by comparing $\tau_{\mathrm{centre}}$ to $\tEsc[\emptyset]$. For any $x \in \Omega^{(k)}_{\emptyset} \backslash \{(0,0,\ldots,0)\}$,
\be \label{IneqDriftCondToyHighDim1}
\E[H(X_{t+1}) | X_{t} = x] \leq H(x) + \frac{2}{3} e^{-C \log(m)} - \frac{1}{3m}. 
\ee 
Since $|H(X_{t+1}) - H(X_{t})| \leq 1$, this inequality implies by the classical  `gambler's ruin' calculation that
\be \label{IneqGamblersRuinToyHighDim}
\min_{x \in \Omega^{(k)}_{\emptyset} }\P_{x}[\tau_{\mathrm{centre}} \leq \tEsc[\emptyset]] \geq 1 - \big( \frac{8m}{3} e^{-C \log(m)} \big)^{k}. \\
\ee 
 For $z \subset [m]$, let $f(z)$ be the unique point in $\Omega_{z}$ that satisfies $H(f(z)) = 0$;  for example, $f(\emptyset) = (0,0,\ldots,0)$. Let $\{X_{t}\}_{t \geq 0}$, $\{Y_{t} \}_{t \geq 0}$ be two Markov chains evolving according to $\tilde{K}$, with $X_{0} = x \in \Omega_{\emptyset}^{(k)}$ arbitrary and $Y_{0} = (0,0,\ldots,0)$. Let $\tEsc[\emptyset]^{(x)}$, $\tEsc[\emptyset]^{(0)}$ be their respective escape times. By inequality \eqref{IneqGamblersRuinToyHighDim}, 
\be \label{IneqIrrelevanceOfStartingPointAppendixC}
\| \mathcal{L}(X_{\tEsc[\emptyset]^{(x)}}) - \mathcal{L}(Y_{\tEsc[\emptyset]^{(0)}}) \|_{\mathrm{TV}} \leq \P_{x}[\tau_{\mathrm{centre}} > \tEsc[\emptyset]] \leq  \big( \frac{8m}{3} e^{-C \log(m)} \big)^{k}.
\ee 
The kernel $\overline{K}_{\mathrm{LL}}$ associated with base kernel $\overline{K}$ and partition $\Omega^{(k)} = \sqcup_{z \subset [m]} \Omega_{z}^{(k)}$ is exactly $\frac{1}{2}$-lazy simple random walk on the hypercube $\{0,1\}^{m}$ (see equation \eqref{EqDefProjChainLL} for the construction of $\overline{K}_{\mathrm{LL}}$). It is straightforward to check via a path-coupling argument (or see the proof of Theorem 15.1 of \cite{LPW09} at inverse-temperature $\beta =0$) that this kernel satisfies the contraction condition
\be 
\max_{x,y \in \{0,1\}^{m}} \frac{W_{d}(\overline{K}_{\mathrm{LL}}(x,\cdot), \overline{K}_{\mathrm{LL}}(y,\cdot))}{d(x,y)} \leq 1 - \frac{1}{m}.
\ee 

Thus, by inequality \eqref{IneqIrrelevanceOfStartingPointAppendixC} the kernel $\tilde{K}$ satisfies inequality \eqref{IneqContractionAssumption} with constants 
\be [EqContAssumptionToyHighDimConc]
\alpha &= 1 - \frac{1}{m} - 2 \big( \frac{8m}{3} e^{-C \log(m)} \big)^{k} \\
\beta &= 2 \big( \frac{8m}{3} e^{-C \log(m)} \big)^{k}.
\ee

We now prove occupation inequalities of the form \eqref{IneqRegularityAssumptions2}. This requires estimates of the escape time $\tEsc[\emptyset]$ and the mixing time $\varphi_{\emptyset}$. By the same calculations that give inequality \eqref{IneqDriftCondToyHighDim1}, we have for $D > 0$ that
\be [IneqDriftIterHighDimToyPre]
\E[e^{D H(X_{t+1})} | X_{t} = x \in \Omega^{(k)}_{\emptyset} \backslash \{(0,0,\ldots,0)\}] &\leq e^{D H(X_{t})} (1 + \frac{2}{3} e^{-C \log(m)} (e^{D} - 1) + \frac{1}{3m} (e^{-D} - 1)) \\
\E[e^{D H(X_{t+1})} | X_{t} = (0,0,\ldots,0)] &\leq \frac{2}{3} e^{-C \log(m)} e^{D}.
\ee 
Setting $D = \frac{C \log(m)}{2}$ and iterating, we have
\be 
\E[e^{\frac{C}{2} H(X_{t}) \log(m)} | X_{0} = x] &\leq (1 - \frac{1}{12m}) \E[e^{\frac{C}{2} H(X_{t-1}) \log(m)}  | X_{0} = x] + \frac{2}{3} e^{-\frac{C}{2} \log(m)}\\
&\leq (1 - \frac{1}{12m})^{t} e^{\frac{C}{2} (\ell-k) \log(m)} + 8m e^{- \frac{C}{2} \log(m)} \label{IneqDriftIterHighDimToy}
\ee 
for all $x \in \Omega_{\emptyset}^{(k)}$. 
%By Markov's inequality, this implies  
%\be 
%\sup_{x \in \Omega^{(k)}} \P[\tEsc[\emptyset] = t] &\leq e^{\frac{C}{2} \ell \log(m)} ((1 - \frac{1}{12m})^{t} e^{\frac{C}{2} (\ell-k) \log(m)} + 8n e^{- \frac{C}{2}\log(m)} ) \\
%&\leq 2 e^{-\frac{Ck}{2} \log(m)}.
%\ee 
Iterating inequality \eqref{IneqDriftIterHighDimToyPre} again and applying Markov's inequality gives
\be
\sup_{x \in \Omega_{\emptyset}^{(k)}} \P_{x} [\tau_{\mathrm{centre}} > t] &\leq \sup_{x \in \Omega_{\emptyset}^{(k)}} \E_{x}[e^{\frac{C}{2} H(X_{t}) \log(m)} \textbf{1}_{\tau_{\mathrm{centre}} > t}] e^{-\frac{C}{2} \log(m)} \\
&\leq  (1 - \frac{1}{12m})^{t} e^{\frac{C}{2} (\ell-k) \log(m)} e^{-\frac{C}{2} \log(m)}. \label{IneqHighDimToyPreCoup1}
\ee 
By Lemma \ref{LemmaDriftConc}, it also gives
\be \label{IneqHighDimToyPreCoup2}
\frac{\pi(f(\emptyset))}{\pi(\Omega_{\emptyset})} \geq 1 - 96 m^{2} e^{-C \log(m)}.
\ee

By inequalities \eqref{IneqHighDimToyPreCoup1} and \eqref{IneqHighDimToyPreCoup2}, $\tilde{K}_{\emptyset}$ satisfies the hypotheses of Lemma \ref{LemmaCoupOnePointSupp} with $T = \frac{\log(2) + \frac{C \log(m)}{2} (\ell-k-1)}{\log(1 - \frac{1}{12m})} m$ and $\epsilon = 96 m^{2} e^{-C \log(m)}$. Thus, for sufficiently large $m$, 
\be 
\varphi_{\emptyset} &\leq  2 \lceil e \,  \frac{\log(2) + \frac{C \log(m)}{2} (\ell-k-1)}{\log(1 - \frac{1}{12m})} m \rceil  \\
&\leq 18 C m (\ell - k + 1) \log(m). \label{IneqToyHighDim1BoundOnVarPhiMax}
\ee 
The symmetry of the problem implies that $\varphi_{\max} = \varphi_{\emptyset}$. This gives the desired upper bound on $\varphi_{\max}$. To obtain a lower bound on $\varphi_{\emptyset}$, we note from inequality \eqref{IneqHighDimToyPreCoup2} that 
\be 
\varphi_{\emptyset} \geq \frac{1}{8} \max_{x \in \Omega_{\emptyset}^{(k)}} \E_{x}[\tau_{\mathrm{centre}}].
\ee 
By considering the number of steps it takes to get to $(0,0,\ldots,0)$, we have $\max_{x \in \Omega_{\emptyset}^{(k)}} \tau_{\mathrm{centre}} \geq m(\ell-k-1)$. By the standard `coupon collector' argument, we also have $\max_{x \in \Omega_{\emptyset}^{(k)}} \E_{x}[\tau_{\mathrm{centre}}] \geq \frac{1}{4} m \log(m)$. Combining these two bounds with inequality \eqref{IneqToyHighDim1BoundOnVarPhiMax},
\be \label{IneqToyHighDimBothBoundVarPhiMix}
18C m (\ell-k+1) \geq \varphi_{\emptyset} \geq \frac{m}{8} \max( \ell-k-1, \frac{1}{4} \log(m)).
\ee 

Finally, we point out that the escape time $\tEsc[\emptyset]$ is often close to $\E[\tEsc[\emptyset] | X_{0} = (0,\ldots,0)]$. In one direction, the symmetry of the problem and the fact that the Markov chain only changes a single coordinate by 1 at every step together imply

\be 
\E[\tEsc[\emptyset] | X_{0} = z \in \Omega_{\emptyset}^{(k)}] \leq \E[\tEsc[\emptyset] | X_{0} = (0,0,\ldots,0)].
\ee 
 Thus, for any $c \in \mathbb{N}$, inequality \eqref{eqn:subexp2} implies
\be \label{IneqConcTEscUpper}
\P[\tEsc[\emptyset] > c \, e \, \E[\tEsc[\emptyset] | X_{0} = (0,0,\ldots,0)] | X_{0}  = z \in \Omega_{\emptyset}^{(k)}] \leq e^{-c}.
\ee 
In the other direction, for all $t > 0$ we have
\be 
\max_{z \in \Omega_{\emptyset}^{(k)}} \P[\tEsc[\emptyset] < t | X_{0} = z] &\leq \P[\tEsc[\emptyset] < t | X_{0} = (0,\ldots,0)] + \max_{z \in \Omega_{\emptyset}^{(k)}}  \P_{z}[\tEsc[\emptyset] < \tau_{\mathrm{centre}}]. 
\ee
This bound, together with inequality \eqref{IneqGamblersRuinToyHighDim}, implies
\be \label{IneqConcTEscLower} 
\max_{z \in \Omega_{\emptyset}^{(k)}} \P[\tEsc[\emptyset] < \frac{1}{8} \E[\tEsc[\emptyset] | X_{0} = (0,0,\ldots,0)] | X_{0} = z] \leq \frac{1}{4} +  \big( \frac{8m}{3} e^{-C \log(m)} \big)^{k}.
\ee 

Thus, inequalities \eqref{IneqConcTEscUpper} and \eqref{IneqConcTEscLower} give for large $m$
\be [IneqConcTEscAll]
\max_{z \in \Omega_{\emptyset}^{(k)}} \P[\tEsc[\emptyset] < \frac{1}{8} \E[\tEsc[\emptyset] | X_{0} = (0,0,\ldots,0)] | X_{0} = z] &\leq \frac{1}{2} \\
\max_{z \in \Omega_{\emptyset}^{(k)}} \P[\tEsc[\emptyset] > 4 \E[\tEsc[\emptyset] | X_{0} = (0,0,\ldots,0)] | X_{0} = z] &\leq \frac{1}{2}. \\
\ee 

We point out that the expected escape time $\E[\tEsc[\emptyset] | X_{0} = (0,0,\ldots,0)]$ is large compared to the mixing time $\varphi_{\emptyset}$.  By the monotonicity of the chain and inequality \eqref{IneqHighDimToyPreCoup2}, it also gives
\be \label{IneqIntermediateHittingLowerBound}
\P[\tEsc[\emptyset] < t | X_{0} = (0,0,\ldots,0)] &\leq  t \max_{0 \leq s \leq t} \P[H(X_{s}) \geq \ell - k | X_{0} = (0,0,\ldots,0) ] \\
&\leq t \max_{0 \leq s \leq t} e^{-\frac{C}{2}( \ell - k) \log(m)} \E[e^{\frac{C}{2} H(X_{s}) \log(m)} | X_{0} = (0,0,\ldots,0) ] \\
&\leq t e^{-\frac{C}{2} (\ell - k) \log(n)}  \big( 96 m^{2} e^{-C \log(m)} e^{\frac{C}{2} (\ell -k)} \big) = 96 t m^{2} e^{-C \log(m)},
\ee 
where the second step is Markov's inequality.
Combining this bound with inequality \eqref{IneqToyHighDimBothBoundVarPhiMix}, we have
\be  \label{TrivialHitMixIneq}
\P[\tEsc[\emptyset] < 64 m \varphi_{\max} | X_{0} = (0,\ldots,0)] \leq \frac{1}{m}.
\ee 
Inequality \eqref{IneqIntermediateHittingLowerBound} also gives
\be 
\P[\tEsc[\emptyset] < \frac{1}{192 m^{2}} e^{C \log(m)} | X_{0} = (0,\ldots,0)] \leq \frac{1}{2},
\ee 
so
\be \label{IneqExpectedHittingTimeLowerBound}
\E_{\emptyset}[\tEsc[\emptyset]] \geq \frac{1}{384 m^{2}} e^{C \log(m)}.
\ee 
Finally, using inequalities \eqref{EqContAssumptionToyHighDimConc} and \eqref{IneqConcTEscAll} and the calculation in Example 12.17 of \cite{LPW09} on the mixing time of simple random walk on the hypercube, we can apply Theorem \ref{ThmContCondHighDim} with constants 
\be 
\alpha &= 1 - \frac{1}{2m},
\,\beta = \frac{1}{m^{3}}, \,D_{\max} = m \\
a_{1} = 16 a_{2} &= 2 \frac{\E[\tEsc[\emptyset] | X_{0} = (0,0,\ldots,0)] \log_{2}(e)}{\varphi_{\emptyset} m } \\
\delta_{1} = \delta_{2} &= \frac{1}{2},\,\,
\overline{\varphi} \leq 2 m \log(m). 
\ee 
The associated value of $T$ may be taken to be $T = 8 m \log(8m)$. By inequality \eqref{TrivialHitMixIneq}, we have that $a_{1}, a_{2} \gg 1$. This completes the proof of inequality \eqref{IneqBigListOfConstantsToyHighDimEx}.

%\section*{Appendix D}
%We prove inequality \eqref{IneqPincenezSymmetry}. Throughout the following problem, we emphasize that the random time $\tau_{\{ n + 1 \}}$ is always defined with respect to the `outer' expectation, not the inner conditional expectation.
%\be
%\E_{1}[\kappa_{1}[T]] &= \E_{1}[ \tau_{\{n+1\}} + \E_{n+1}[ (T + 1- S) - \kappa_{2}(T  - S) | S = \tau_{\{n+1\}}]] \\
%&= \E_{1}[\tau_{\{ n + 1 \}} + (T +1 - \tau_{\{n+1\}}) - \E_{n+1}[ \kappa_{2}(T  - S) |  S =\tau_{\{n+1\}}]]  \\
%&= T +1 - \E_{1}[ \E_{n+1}[ \kappa_{2}(T - S) | S = \tau_{\{n+1\}}]] \\
%&= T +1 - \E_{1}[ \E_{1}[\kappa_{1}(T - S) | S= \tau_{\{n+1\}}]].
%\ee
%Thus,
%\be 
%\E_{1}[\kappa_{1}(T)] + \E_{1}[ \E_{1}[\kappa_{1}(T - S) | S =\tau_{\{n+1\}}]] = T + 1.
%\ee 
%Since $\E_{1}[\kappa_{1}(s)] \geq \E_{1}[\kappa_{1}(t)]$ for any $s \geq t$, this implies
%\be 
%\E_{1}[\kappa_{1}(T)] + \E_{1}[\kappa_{1}(T)] \geq \E_{1}[\kappa_{1}(T)] + \E_{1}[ \E_{1}[\kappa_{1}(T - S) | S =\tau_{\{n+1\}}]] = T + 1,
%\ee 
%which is exactly the desired inequality.
\end{document}